\documentclass{amsart}
\usepackage{amscd}
\usepackage{amssymb}
\usepackage[all, knot]{xy}
\usepackage{bbm}
\usepackage[top=1.2in, bottom=1.2in, left=1.2in, right=1.2in]{geometry}
\xyoption{all}
\xyoption{arc}
\usepackage{hyperref}
\usepackage[mathscr]{euscript}

\def\res{\on{res}}
\def\Gres{\on{res}^G}

\def\Kos{\operatorname{Kos}}

\def\Thick{\operatorname{Thick}}
\def\tstar{{\tilde\star}}

\def\del{\partial}

\def\cube#1#2#3#4#5#6#7#8{
& #5 \ar[rr] \ar[dl] \ar@{-}[d] && #6 \ar[dd] \ar[dl] \\
#1 \ar[rr] \ar[dd]  & \ar[d] & #2 \ar[dd] \\
& #7 \ar@{-}[r] \ar[dl] & \ar[r] & #8 \ar[dl] \\
#3 \ar[rr] && #4 \\
}

\def\S{\Sigma}

\def\smsh{\wedge}

\def\cE{\mathcal E}

\def\cA{\mathcal A}
\def\cB{\mathcal B}
\def\cC{\mathcal C}

\def\cD{\mathcal D}

\def\dm{\operatorname{dim}}

\def\Sing{\operatorname{Sing}}

\def\sExt{\sExt}

\def\End{\operatorname{End}}

\def\Spec{\operatorname{Spec}}

\def\bu{\bullet}

\def\into{\hookrightarrow}
\def\onto{\twoheadrightarrow}

\def\a{\alpha}
\def\b{\beta}

\def\e{\varepsilon}

\newcommand{\A}{\mathbb{A}}

\newcommand{\G}{\mathbb{G}}

\newcommand{\R}{{\mathbb{R}}}

\newcommand{\C}{\mathbb{C}}
\newcommand{\Z}{\mathbb{Z}}

\newcommand{\fp}{{\mathfrak p}}
\newcommand{\fm}{{\mathfrak m}}

\numberwithin{equation}{section}

\theoremstyle{plain} 
\newtheorem{thm}[equation]{Theorem}
\newtheorem{thm-conj}[equation]{Theorem-Conjecture}
\newtheorem{defn-conj}[equation]{Definition-Conjecture}

\newtheorem*{introthm*}{Theorem}

\newtheorem{cor}[equation]{Corollary}
\newtheorem{lem}[equation]{Lemma}

\newtheorem{prop}[equation]{Proposition}

\newtheorem{conj}[equation]{Conjecture}

\theoremstyle{definition}
\newtheorem{defn}[equation]{Definition}

\theoremstyle{remark}
\newtheorem{rem}[equation]{Remark}

\def\Perf{\operatorname{Perf}}

\newcommand{\Hom}{\operatorname{Hom}}

\newcommand{\supp}{\operatorname{supp}}

\newcommand{\xra}[1]{\xrightarrow{#1}}
\newcommand{\xla}[1]{\xleftarrow{#1}}

\newcommand{\id}{\operatorname{id}}

\newcommand{\odd}{\operatorname{odd}}

\def\L{\Lambda}

\def\ev{{\mathrm{even}}}
\def\odd{{\mathrm{odd}}}
\def\t{\theta}

\def\s{\sigma}
\def\d{\delta}

\def\tr{\operatorname{tr}}
\def\trace{\operatorname{trace}}
\def\str{\operatorname{str}}

\def\n{\nabla}

\def\and{ \text{ and } }

\def\can{\mathrm{can}}
\def\Der{\operatorname{Der}}

\def\op{{\mathrm{op}}}

\def\G{\Gamma}

\def\nat{\natural}

\def\on{\operatorname}

\def\Ob{\on{Ob}}

\def\Fold{\on{Fold}}

\begin{document}

\begin{abstract}
We prove a conjecture of Shklyarov concerning the relationship between K. Saito's higher residue pairing and a certain pairing on the periodic cyclic homology of matrix factorization categories. Along the way, we give new proofs of a result of Shklyarov (\cite[Corollary 2]{Shklyarov:higherresidues}) and Polishchuk-Vaintrob's Hirzebruch-Riemann-Roch formula for matrix factorizations (\cite[Theorem 4.1.4(i)]{PV}).
\end{abstract} 

\title{A proof of a conjecture of Shklyarov}

\author{Michael K. Brown}
\address{Department of Mathematics, University of Wisconsin-Madison, WI 53706-1388, USA}
\email{mkbrown5@wisc.edu}

\author{Mark E. Walker}
\address{Department of Mathematics, University of Nebraska-Lincoln, NE 68588-0130, USA}
\email{mark.walker@unl.edu}

\thanks{MB gratefully acknowledges support from the National Science Foundation (award DMS-1502553), and MW
gratefully acknowledges support from the Simons Foundation (grant \#318705) and the National Science Foundation (award DMS-1901848).} 
\maketitle

\tableofcontents

\section{Introduction}
Let $Q = \C[x_1, \dots, x_n]$, and let $\fm$ denote the maximal ideal $(x_1, \dots, x_n) \subseteq Q$. 
Fix $f \in \fm$, and assume the only singular point of the associated morphism $f : \Spec(Q) \to \A^1_\C$ is $\fm$. 
Let $mf(Q, f)$
denote the differential $\Z/2$-graded category of matrix factorizations of $f$; see Section \ref{mf} for the definition of $mf(Q, f)$.
Shklyarov proves in \cite[Theorem 1]{Shklyarov:higherresidues} that a certain pairing on the periodic cyclic homology of $mf(Q, f)$
coincides, up to a constant factor $c_f$ (which possibly depends on $f$), with K.~Saito's higher residue pairing, via the
Hochschild-Kostant-Rosenberg (HKR) isomorphism. Shklyarov conjectures in
\cite[Conjecture 3] {Shklyarov:higherresidues} that $c_f = (-1)^{\frac{n(n+1)}{2}}$. The main goal of this paper is to prove this conjecture.  

We begin by discussing Shklyarov's conjecture in more detail.

\subsection{Background on Shklyarov's conjecture}
Let $HN(mf(Q, f))$ denote the negative cyclic complex of $mf(Q, f)$, and let $HN_*(mf(Q, f))$ denote its homology. See, for instance, \cite[Section 3]{BW} for the definition of the negative cyclic complex of a dg-category. The dg-category $mf(Q,f)$ is \emph{proper}, i.e. each cohomology group of the ($\Z/2$-graded) morphism complex of any two objects is a
finite dimensional $\C$-vector space. As with any such dg-category, there is a canonical pairing of $\Z/2$-graded $\C$-vector spaces 
$$
K_{mf} : HN_*(mf(Q,f)) \times HN_*(mf(Q,f)) \to \C[[u]],
$$
where $u$ is an even degree variable. The pairing $K_{mf}$ is defined exactly as in \cite[page 184]{Shklyarov:higherresidues}, but with periodic cyclic homology $HP_*$ replaced with $HN_*$ and $\C((u))$ replaced with $\C[[u]]$.  We note that $K_{mf}$ is \emph{$\C[[u]]$-sesquilinear}; that is, for any $\a, \b \in HN_*(mf(Q,f))$ and $g \in \C[[u]]$, we have
$$
K_{mf}( g(u) \cdot \a, \b ) = g(u) K_{mf}( \a, \b) = K_{mf}( \a , g(-u) \cdot \b).
$$

It follows from work of Segal \cite[Corollary 3.4]{segal} and Polishchuk-Positselski \cite[Section 4.8]{PP} that there is a quasi-isomorphism
\begin{equation}
\label{HKRif}
I_f: HN(mf(Q,f)) \xra{\simeq} (\Omega_{Q/\C}^\bu[[u]], ud - df),
\end{equation}
which generalizes the classical Hochschild-Kostant-Rosenberg (HKR) theorem.
The target of $I_f$ is called the \emph{twisted de Rham complex}, and
it is a $\Z/2$-graded complex indexed by setting $\Omega^m_{Q/\C}$ to have (homological) degree $m$ and $u$ to have degree $-2$.
(Since the twisted de Rham complex is $\Z/2$-graded, 
we could just as well say $\Omega^m$ has degree $-m$ and $u$ has degree $2$. Note that the map $ud$ has
degree $-1$ whereas $df$ has degree $1$, but since this is regarded as a $\Z/2$-graded complex, there is no problem.)
In particular, we have an isomorphism
$$
I_f: HN_n(mf(Q,f)) \xra{\cong} H^{(0)}_f,
$$
where
$$
H_f^{(0)} := H_n(\Omega^\bu_{Q/\C}[[u]], u d - df) = \frac{\Omega^n_{Q/\C}[[u]]}{(u d -df) \cdot  \Omega^{n-1}_{Q/\C}[[u]]}.
$$
In \cite{saito}, K. Saito equips the $\C[[u]]$-module $H_f^{(0)}$ with a pairing
$$
K_f : H^{(0)}_f \times H^{(0)}_f \to \C[[u]]
$$
known as the {\em higher residue pairing}. Shklyarov has proven the following result concerning the relationship between the 
canonical pairing and the higher residue pairing under the HKR isomorphism:

\begin{thm}[\cite{Shklyarov:higherresidues}, Theorem 1] For each polynomial $f$ as above, there is a constant $c_f \in \C$ (possibly depending on $f$) such that the diagram
\begin{equation} \label{E122b}
\xymatrix{
HN_n(mf(Q,f))^{\times 2} \ar[rr]^-{I_f \times I_f}_-\cong \ar[dr]_-{c_{f} \cdot u^n \cdot K_{mf}} &&  \left(H_f^{(0)}\right)^{\times 2} \ar[dl]^-{K_f} \\
& \C[[u]] \\
}
\end{equation}
commutes.
\end{thm}

Moreover, Shklyarov makes the following prediction:

\begin{conj}[\cite{Shklyarov:higherresidues}, Conjecture 3]
\label{conj:s}
For any $f$, $c_f = (-1)^{n(n+1)/2}$. 
\end{conj}

\subsection{Outline of the proof of Conjecture \ref{conj:s}}

The constant $c_f$ can be determined from a related, but simpler, pairing on $HH_*(mf(Q, f))$, the Hochschild homology of $mf(Q, f)$. We recall that, for any dg-category $\cC$, there is a short exact sequence
\begin{equation}
\label{modu}
0 \to HN(\cC) \xra{\cdot u} HN(\cC) \to HH(\cC) \to 0
\end{equation}
of complexes. It follows, for instance from (\ref{HKRif}),
that $HN_*(mf(Q, f))$ and $HH_*(mf(Q, f))$ are concentrated in degree $n$ (mod 2).  The long exact sequence in homology induced by (\ref{modu}) therefore induces an isomorphism
\begin{equation}
\label{moduH}
HN_*(mf(Q, f))/u\cdot HN_*(mf(Q, f)) \xra{\cong} HH_*(mf(Q, f)).
\end{equation}
The pairing $K_{mf}$ determines a well-defined pairing modulo $u$, which we write, via (\ref{moduH}), as
$$
\eta_{mf} : HH_*(mf(Q, f)) \times HH_*(mf(Q, f)) \to \C.
$$
The isomorphism $I_f$ is $\C[[u]]$-linear and, upon
setting $u = 0$, it induces an isomorphism
$$
I_f(0): HH_n(mf(Q,f)) \xra{\cong} H_n(\Omega^\bu_{Q/\C}, -df).
$$

The higher residue pairing $K_f$ has the form
$$
K_f\left(\omega + \sum_{j \geq 1} \omega_j u^j, 
\omega' + \sum_{j \geq 1} \omega_j' u^j\right)
= \langle \omega, \omega' \rangle_{\on{res}} u^n + \text{higher order terms},
$$
where $\langle \omega, \omega' \rangle_{\on{res}}$ is the classical residue pairing determined by the partial derivatives of $f$. It is 
defined algebraically as
$$
\langle g \cdot dx_1 \cdots dx_n , h \cdot dx_1 \cdots dx_n \rangle_{\on{res}} = 
\res \left[\frac{g h \cdot dx_1 \cdots dx_n}{\frac{\del f}{\del x_1}, \dots, \frac{\del f}{\del x_n}}\right],
$$ 
where the right-hand side is Grothendieck's residue symbol.

Thus, upon dividing the maps in diagram (\ref{E122b}) by $u^n$ and setting $u = 0$, we obtain the commutative triangle
\begin{equation} \label{E122}
\xymatrix{
HH_n(mf(Q,f)) \times HH_n(mf(Q,f)) \ar[rr]^-{I_f(0) \times I_f(0)}_-\cong \ar[dr]_-{c_{f} \eta_{mf}} &&  
\frac{\Omega^n_{Q/\C}}{ df \smsh \Omega^{n-1}_{Q/\C}} \times \frac{\Omega^n_{Q/\C}}{ df \smsh \Omega^{n-1}_{Q/\C}}
\ar[dl]^-{\langle -,- \rangle_{\on{res}}}\\
& \C. \\
}
\end{equation}
Since $I_f(0)$ is an isomorphism, and the residue pairing is non-zero,
 the value of $c_f$ is uniquely determined by the commutativity of \eqref{E122}. 
 
In this paper, we re-establish the commutativity of diagram \eqref{E122} using techniques that differ from 
those used by Shklyarov. Our method results in an explicit
calculation of $c_f$:

\begin{thm} \label{introthm1} Shklyarov's Conjecture holds: that is, for any $f$ as above, 
$$
c_f = (-1)^{n(n+1)/2}.
$$
\end{thm}

In fact, we prove the commutativity of diagram \eqref{E122}, and Theorem \ref{introthm1}, in the case where $Q$ is an essentially smooth algebra over a characteristic 0 field $k$,
$\fm$ is a $k$-rational maximal ideal, and $f \in
\fm$ is such that $\fm$ is the only singularity of the morphism $f: \Spec(Q) \to \A^1_k$. 
The special case $k = \C$, $Q = \C[x_1, \dots, x_n]$, and $\fm = (x_1, \dots, x_n)$ yields Shklyarov's conjecture.

The general outline of our proof is summarized by the diagram
\begin{equation}
\label{outline}
\xymatrix{
HH_n(mf(Q,f)) \times HH_n(mf(Q,f))  \ar[d]^-{\id \times \Psi} \ar[rr]^-{I_f(0) \times I_f(0)}_-\cong && H_n(\Omega_{Q/k}^\bu, -df) \times H_n(\Omega_{Q/k}^\bu, -df)
\ar[d]^-{\id \times (-1)^n}\\
HH_n(mf(Q,f)) \times HH_n(mf(Q,-f)) \ar[rr]^-{I_f(0) \times I_{-f}(0)}_-\cong 
\ar[d]^-\star  && H_n(\Omega_{Q/k}^\bu, -df) \times H_n(\Omega_{Q/k}^\bu, df)
\ar[d]^-{\smsh} \\
HH_{2n}(mf^\fm(Q_{\fm},0))  \ar[dr]_-{(-1)^{n(n+1)/2} \trace} \ar[rr]^-{\e} && H_{2n} \R\Gamma_\fm (\Omega_{Q_{\fm}/k}^\bu) \ar[dl]^-{\on{res}} \\
& k. \\
}
\end{equation}
The map $\Psi$ is induced by taking $Q$-linear
duals; $\star$ is induced by a K\"unneth map followed by the tensor product of matrix factorizations; $\trace$ is
defined in Section \ref{traceresidue}; $\on{res}$ is Grothendieck's residue map; $\smsh$ is induced by exterior multiplication of differential forms, using that
the complexes  $(\Omega^\bu_{Q/k}, \pm f)$ are supported on $\{\fm\}$; and the map $\e$ is an HKR-type map.
We prove: 

\begin{enumerate}
\item the diagram commutes (Lemma \ref{If}, Lemma \ref{lem1129}, Corollary \ref{cor1129}, and Theorem \ref{thm112}),

\item the composition along the left side of this diagram is the canonical pairing $\eta_{mf}$ (Lemma \ref{lem1119t}), and 

\item the composition along the right side of this diagram is the residue pairing $\langle -,- \rangle_{\on{res}}$ (Proposition \ref{prop1129}).
\end{enumerate}

Finally, in Section \ref{PV}, we use some of our results to give a new proof Polishchuk-Vaintrob's Hirzebruch-Riemann-Roch theorem for matrix factorizations (\cite[Theorem 4.1.4(i)]{PV}).

We note that a result closely related to the commutativity of (\ref{E122}) was also proven by Polischuk-Vaintrob (\cite[Corollary 4.1.3]{PV}). More precisely, they prove that the residue pairing on $H_n(\Omega_{Q/k}^\bu, -df) \times H_n(\Omega_{Q/k}^\bu, df)$ and the canonical pairing on $HH_n(mf(Q,f)) \times HH_n(mf(Q,-f))$ coincide up to multiplication by $(-1)^{(n-1)n/2}$ via an isomorphism 
$$
\gamma: HH_n(mf(Q, f)) \xra{\cong} H_n(\Omega^\bullet_{Q/k}, df)
$$ 
described in \cite[(2.28)]{PV}. Combining this result of Polishchuk-Vaintrob with our Theorem \ref{introthm1} and the non-degeneracy of the residue pairing, we conclude that, if $\alpha, \alpha' \in HH_n(mf(Q, f))$, 
\begin{equation}
\label{PViso}
\langle \gamma(\alpha), \gamma(\alpha') \rangle_{\on{res}} = \langle I_f(0)(\alpha), I_f(0)(\alpha')\rangle_{\on{res}}.
\end{equation}
If one could prove (\ref{PViso}) directly, one could simply combine \cite[Corollary 4.1.3]{PV} with the commutativity of the top square of diagram (\ref{outline}) to quickly prove Shkylarov's conjecture. But we believe there is no way to prove (\ref{PViso}) without going through Theorem 1.8.

\vskip\baselineskip
\noindent {\bf Acknowledgements.}
We thank Srikanth Iyengar for helpful remarks concerning the proof of Lemma \ref{lem529}.
\section{Generalities on Hochschild homology for curved dg-categories}

We review some background on Hochschild homology of curved dg-categories and establish some new results concerning pairings
of such. Throughout this section, $k$ is a field, and ``graded'' means $\G$-graded for $\G \in \{\Z ,\Z/2\}$. We will eventually focus on the case $\G = \Z/2$. 

\subsection{Hochschild homology of curved dg-categories}
We refer the reader to \cite[Section 2.1]{BW} for the definition of a curved differential $\G$-graded category (henceforth referred to as a cdg-category). 
Recall that a cdg-category with just one object is a curved differential $\G$-graded algebra
(cdga).

For a cdg-category $\cC$ whose objects form a set, define $HH(\cC)^\nat$ to be the $\G$-graded $k$-vector space given by the direct sum totalization of the $\Z-\G$-bicomplex
which, in $\Z$-degree $n$, is the $\G$-graded $k$-vector space
$$
\bigoplus_{X_0, \dots, X_n \in \cC}
\Hom(X_1, X_0) \otimes_k \Sigma \Hom(X_2, X_1) \otimes_k \cdots \otimes_k \Sigma \Hom(X_{n}, X_{n-1}) \otimes_k \Sigma \Hom(X_0, X_n).
$$
When $\cC$ is \emph{essentially small}, so that the isomorphism classes of objects in the $\G$-graded category
underlying $\cC$ form a set (see \cite[Section 2.6]{PP}), we define $HH(\cC)^\nat$ by first replacing
$\cC$ with a full subcategory consisting of a single object from each isomorphism class. From now on, we will tacitly assume all of our cdg-categories are essentially small. 
Given $\a_i \in \Hom(X_{i+1}, X_i)$ for $i = 0, \dots, n$ (with $X_{n+1} = X_0$), we write 
$\a_0[\a_1| \cdots | \a_n]$ for the element $\a_0 \otimes s \a_1 \otimes \cdots \otimes s \a_n$ of $HH(\cC)^\nat$.  

The \emph{Hochschild complex of $\cC$}, denoted $HH(\cC)$, is the above graded $k$-vector space equipped with the differential $b := b_2 + b_1 + b_0$, where $b_2, b_1, b_0$ are defined as in \cite[Section 3.1]{BW}. Roughly, $b_2$ is the classical Hochschild differential induced by the composition law in $\cC$, 
$b_1$ is induced by the differentials of $\cC$, and $b_0$ is induced by the curvature elements of
$\cC$.  When $\cC$ has just one object with trivial curvature, then $\cC$ is a dga, and the maps $b_2$ and $b_1$ are the classical ones (and $b_0 = 0$
in this case). 

We will also need ``Hochschild homology of the second kind'', as introduced by Polishchuk-Positselski in \cite{PP} and by C{\u a}ld{\u a}raru-Tu in \cite{CT}; the latter authors call this theory ``Borel-Moore Hochschild homology". Define
$HH^{II}(\cC)^\nat$ to be the $\G$-graded $k$-vector space
given as the direct {\em product} totalization of the above bicomplex.  Equivalently, $HH^{II}(\cC)^\nat$ is the
completion of $HH(\cC)^\nat$ under the topology determined by the
evident filtration. Since $b$ is continuous for this topology, it
induces a differential on $HH^{II}(\cC)^\nat$, which we also write as $b$, and we write
$HH^{II}(\cC)$ for the resulting chain complex. 

\subsection{The K\"unneth map for Hochschild homology of cdga's}
For a cdga $\cA = (A, d_A, h_A)$, we have
$$
HH(A)^\nat = A \otimes_k T(\Sigma A),
$$
where, for any graded $k$-vector space $V$, $T(V) = \bigoplus_{n \ge 0} V^{\otimes n}$.
Recall that $T(V)$ is a commuative $k$-algebra under the {\em shuffle product}:
$$
(v_1 \otimes \cdots \otimes v_p) \bu (v_{p+1} \otimes \cdots \otimes v_{p+q}) 
= \sum_\s \pm v_{\s(1)} \otimes \cdots \otimes v_{\s(p+q)},
$$
where $\s$ ranges over all $(p,q)$-shuffles. The sign is given by the usual rule for permuting homogeneous elements in a product.

Since $A$ is also an algebra, $HH(A)^\nat$ has an algebra structure, whose multiplication
rule will be written as 
$$
- \star -: HH(\cA)^\nat \otimes_k HH(\cA)^\nat \to HH(\cA)^\nat.
$$
It is given explicitly as 
$$
x [a_1 | \cdots | a_p] \star y [a_{p+1} | \cdots | a_{p+q}] =
\sum_{\sigma} \pm xy [a_{\s(1}) | \cdots | a_{\s(p+q)}].
$$
Note that the canonical inclusion $T(\S A) \into HH(A)^\nat$ lands in the center of $HH(A)^\nat$ for the $\star$ multiplication.

If $\cB = (B, d_B, h_B)$ is another cdga, the tensor product of $\cA$ and $\cB$ is defined to be 
$$
\cA \otimes_k \cB = (A \otimes_k B, d_A \otimes 1 + 1 \otimes d_B, h_A \otimes 1 + 1 \otimes h_B).
$$
We define the \emph{K\"unneth map}
$$
- \tilde{\star} -: HH(\cA)^\nat \otimes_k HH(\cB)^\nat \to HH(\cA \otimes_k\cB)^\nat
$$
to be the composition of the tensor product of the maps induced by the canonical inclusions $HH(\cA)^\nat \into HH(\cA \otimes \cB)^\nat$
and $HH(\cB)^\nat \into HH(\cA \otimes \cB)^\nat$ with the $\star$ product for $\cA \otimes_k \cB$.
The $\star$ product on $HH(\cA)^\nat$ can be recovered from the K\"unneth map by setting $\cB = \cA$: the $\star$ product coincides with the composition
$$
HH(\cA)^\nat \otimes_k HH(\cA)^\nat \xra{\tilde{\star}} HH(\cA \otimes_k \cA)^\nat \xra{\mu_*} HH(\cA)^\nat ,
$$
where 
$$
\mu_*: (A \otimes_k A) \otimes_k T(\S(A \otimes A)) \to A \otimes_k T(\S A)
$$
is induced by the multiplication map $\mu: A \otimes A \to A$. 

It is important to note that, for an algebra $A$, the $\star$ product
does not, in general, make $HH(A)$ into a dga, since $b_2$ is not a derivation for the $\star$ multiplication unless $A$ is
commutative. But, $b_2$ is a derivation for the K\"unneth map; see Lemma \ref{tedious}.

The $\star$ product does behave well with respect to $b_1$. In detail, recall that the tensor algebra functor $T( - )$ sends $\G$-graded complexes of $k$-vector
spaces to differential $\G$-graded algebras under the shuffle product.
Let $d_T$ denote the differential on $T(\Sigma A)$ induced from the differential $\Sigma d$ on
$\Sigma A$. Then $(T(\Sigma A), \bu, d_T)$ is a dga, where $\bu$ is the shuffle product. By examining the explicit formula for $b_1$, we see that
$$
b_1 = d_A \otimes 1 + 1 \otimes d_T.
$$
In other words, $(HH(\cA)^\nat, \star, b_1)$ is a dga, and it is given as a tensor product of dga's:
$$
(HH(\cA)^\nat, \star, b_1) = (A, \cdot , d_A) \otimes (T(\Sigma A), \bu, d_T),
$$
where $\cdot$ is the multiplication rule for $A$.

If $z$ is an element of $A$ of even degree, then we have
$$
1 [z] \star a_0[a_1 | \cdots | a_n] =
\sum_i (-1)^{|a_0| + |a_1| + \cdots + |a_i| - i}  a_0[a_1| \cdots |a_i| z| a_{i+1} | \cdots |a_n].
$$
In particular, the component $b_0$ of the differential in $HH(\cA)$ is given by 
\begin{equation} \label{1029a}
b_0 = 1 [h] \star -.
\end{equation}
Since $1 [h]$ is a central element of $(A \otimes T(\Sigma A), \star)$ of odd degree, it follows that 
\begin{equation} \label{1029b}
b_0(-) \star - = b_0 (- \star -) = \pm - \star b_0(-).
\end{equation}

The $\star$ product extends to $HH^{II}$ since it is continuous for the topology on $HH$  whose completion gives $HH^{II}$. 

\subsection{Functoriality of $HH^{II}$ using the shuffle product}
We recall that a morphism $\cA = (A,d_A,h_A) \to \cB = (B, d_B,h_B)$ of cdga's is given by a pair $\phi = (\rho, \b)$, with $\rho: A \to B$ a morphism of
$\G$-graded $k$-algebras and $\b \in B$ a degree one element, such that 
\begin{itemize}
\item $\rho(d(a)) - d'(\rho(a)) = [\b, \rho(a)]$ for all $a \in A$, and
\item $\rho(h)  = h' + d'(\b) + \b^2$.
\end{itemize}
Such a morphism is called {\em strict} if $\b = 0$.

A strict morphism $\phi$ induces maps
$$
\phi_*: HH(\cA) \to HH(\cB) \and \phi_*: HH^{II}(\cA) \to HH^{II}(\cB)
$$
given by
$$
\phi_*(a_0 [a_1| \dots| a_n])  = \rho(a_0) [\rho(a_1) | \cdots | \rho(a_n)].
$$
A non-strict morphism $\phi$ does not, in general, induce a map on Hochschild homology, but it does induce a map
$$
\phi_*: HH^{II}(\cA) \to HH^{II}(\cB)
$$
given by sending $a_0 [a_1| \dots| a_n]$ to 
\begin{equation} 
\sum_{i_0, \dots, i_n \geq 0} (-1)^{i_0 + \cdots + i_n} \rho(a_0) [\underbrace{\b | \cdots | \b}_{i_0 \text{ copies}} | \rho(a_1) |\underbrace{\b | \cdots | \b}_{i_1 \text{ copies}}| \rho(a_2) | \cdots | \rho(a_n) | \underbrace{\b | \cdots | \b}_{i_n \text{ copies}}].
\end{equation}

We next show how $\phi_*$ may also be defined using the $\star$ product. Suppose $b \in B$ is a degree $1$ element, and
let $\exp(1[b])$ denote the degree $0$, central element of the algebra $(HH^{II}(B)^\nat, \star)$
given by evaluating the power series for the exponential function
at $1[b]$:
$$
\begin{aligned}
\exp(1[b]) & = 
1 + 1[b] + \frac{1}{2!} (1[b] \star 1[b]) + \frac{1}{3!} (1[b]\star 1[b]\star 1[b]) + \cdots \\
& = 1 + 1[b] + 1[b|b]  + 1[b|b|b] + \cdots. \\
\end{aligned}
$$
The signs are correct, since $s(b) \in T(\Sigma B)$ has even degree. We have:
$$
\begin{aligned}
\exp(1[b]) \star (b_0 [b_1| \dots| b_n]) & = (b_0 [b_1| \dots| b_n])\star \exp(1[b]) \\
& =
\sum_{i_0, \dots, i_n \geq 0} b_0 [\underbrace{b | \cdots | b}_{i_0 \text{ copies}} | b_1 |\underbrace{b | \cdots | b}_{i_1 \text{ copies}}| b_2| \cdots |
b_n |  \underbrace{b | \cdots | b}_{i_n \text{ copies}}]. \\
\end{aligned}
$$
By comparing formulas, we see that
\begin{equation} \label{1029v}
\phi_* = \exp(1[-\b]) \star \rho_*.
\end{equation}
That is, 
$$
\phi_*(a_0 [a_1| \dots| a_n])  = \exp(1[-\b]) \star  \rho(a_0) [\rho(a_1) | \cdots | \rho(a_n)]
= \rho(a_0) [\rho(a_1) | \cdots | \rho(a_n)] \star \exp(1[-\b]).
 $$


\subsection{The K\"unneth map for Hochschild homology of cdg-categories}
\label{kunnethcdgc}
For a pair of cdg-categories 
$\cC$ and $\cD$, we write $\cC \otimes_k \cD$ for the  cdg-category whose objects are ordered pairs $(C, D)$ with $C \in \cC$ and
$D \in \cD$ and such that
$$
\Hom((C,D), (C', D')) = \Hom_\cC(C,C') \otimes_k \Hom_\cD(D, D'),
$$
with differentials given in the standard way for a tensor product.  The composition rules are the evident ones, and the curvature elements are defined by
$$
h_{(C,D)} = h_C \otimes \id_D + \id_C \otimes h_D.
$$
Note that, if $\cA = (A, d_A, h_A)$ and $\cB = (B, d_B, h_B)$ are cdga's, then this construction specializes to the construction given above:
$$
\cA \otimes_k \cB = (A \otimes_k B, d_A \otimes \id_B + \id_A \otimes d_B, h_A \otimes \id_B + \id_A \otimes h_B).
$$

We define the {\em K\"unneth map} for the cdg-categories $\cC$ and $\cD$ to be the map
$$
- \tstar -: HH(\cC)^\nat \otimes_k HH(\cD)^\nat \to HH(\cC \otimes_k \cD)^\nat
$$
given by 
$$
c_0[c_1 | \cdots | c_m] \tstar d_0 [d_1| \cdots |d_n] = \sum_\sigma \pm c_0 \otimes d_0[ e_{\sigma(1)} | \cdots | e_{\sigma(m+n)}],
$$
where $\sigma$ ranges over all $(m,n)$-shuffles, and 
$$
e_i := 
\begin{cases}
c_i \otimes \id, & \text{if $1 \leq i \leq m$, and} \\
\id \otimes d_{i-m} , & \text{if $m+1 \leq i \leq m+n$.}
\end{cases}
$$
This map extends to $HH^{II}( - )^\nat$:
$$
- \tstar -: HH^{II}(\cC)^\nat \otimes_k HH^{II}(\cD)^\nat \to HH^{II}(\cC \otimes \cD)^\nat.
$$

\begin{rem} There does not seem to be an analogue of the $\star$ product for a general cdg-category. The issue is that, in general, there is no ``diagonal map"
$$
\cC \otimes_k \cC \to \cC.
$$
\end{rem}

\begin{lem} 
\label{tedious}
For any two cdg-categories $\cC$ and $\cD$, the diagram
$$
\xymatrix{ 
HH(\cC)^\nat \otimes_k HH(\cD)^\nat \ar[d]^-{b_i \otimes \id + \id \otimes b_i} \ar[r]^-{-\tstar -}& HH(\cC \otimes_k \cD)^\nat \ar[d]^-{b_i}\\
HH^{II}(\cC)^\nat \otimes_k HH(\cD)^\nat \ar[r]^-{- \tstar -}& HH(\cC \otimes_k \cD)^\nat
}
$$
commutes for $i = 0, 1$, and $2$, and similarly for $HH^{II}( - )^\nat$. In particular,
$$
- \tstar -: HH(\cC) \otimes_k HH(\cD) \to HH(\cC \otimes_k \cD)
$$
and
$$
- \tstar -: HH^{II}(\cC) \otimes_k HH^{II}(\cD) \to HH^{II}(\cC \otimes_k \cD)
$$
are chain maps. 
\end{lem}

\begin{proof} 
This follows from the definitions by a routine check.
\end{proof}

\subsection{Naturality of the K\"unneth map}
We recall that a morphism $\cA \to \cB$ of cdg-categories is a pair $\phi = (F, \beta)$, where $F : \cA \to \cB$ is a morphism of categories enriched in $\G$-graded $k$-vector spaces, and $\beta$ is an assignment to each object $X$ of $\cA$ a degree 1 element $\beta_X \in \End_{\cB}(F(X))$. The pair $(F, \beta)$ is required to satisfy:
\begin{itemize}
\item For all $X, Y \in \on{Ob}(\cA)$ and $f \in \Hom_{\cA}(X, Y)$, 
$$
F(\delta(f)) = \delta(F(f)) + \b_Y \circ F(f) - (-1)^{|f|} F(f) \circ  \b_X, 
$$
where $\delta$ is the differential on  $\Hom_{\cA}(X, Y)$; and 
\item for all $X \in \Ob(\cA)$, 
$$
F(h_X) = h_{F(X)} + \delta(\b_X) + \b_X^2.
$$
\end{itemize}
$\phi$ is called \emph{strict} if $\b_X = 0$ for all $X$. 

\begin{lem} \label{lem112}
\label{key}
  Suppose $\cA, \cA', \cB, \cB'$ are curved differential $\G$-graded categories, and $\phi = (F, \beta): \cA \to \cB$,
  $\phi' = (F', \beta'): \cA' \to \cB'$ are morphisms of such. Then
  
  \begin{enumerate}
    \item
      $\phi \otimes \phi' := (F \otimes F', \beta \otimes 1 + 1 \otimes \beta')$ is a morphism from $\cA \otimes_k \cA'$ to $\cB \otimes_k
      \cB'$, and, if $\phi$ and $\phi'$ are strict morphisms, then so is $\phi \otimes \phi'$;

    \item the diagram
      $$
      \xymatrix{
        HH^{II}(\cA) \otimes_k HH^{II}(\cA') \ar[d]^-\tstar \ar[rr]^-{(\phi)_* \otimes (\phi')_*} && HH^{II}(\cB) \otimes_k HH^{II}(\cB') \ar[d]^-\tstar \\
        HH^{II}(\cA \otimes_k \cA') \ar[rr]^-{(\phi \otimes \phi')_*} && HH^{II}(\cB \otimes_k \cB')  \\
      }
      $$
      commutes; and
     
    \item if $\phi$ and $\phi'$ are strict morphisms, the corresponding diagram involving ordinary Hochschild homology commutes.
\end{enumerate}
    \end{lem}

\begin{proof} The proof of (1) is a routine check, and (3) is an immediate consequence of (2). 
For (2), to simplify the notation,
we assume the cdg-categories involved are cdg-algebras; the proof of the general claim is notationally more complicated
but essentially the same. 
Write $\phi = (\rho, \b)$, $\phi' = (\rho', \b')$, so that, by (\ref{1029v}),
$$
  \phi_* = \exp(1[-\b]) \star \rho_* \and
  \phi'_* = \exp(1[-\b']) \star \rho'_*.
  $$

Let $\iota: HH^{II}(\cA) \into HH^{II}(\cA \otimes_k \cA')$ and
$\iota': HH^{II}(\cA') \into HH^{II}(\cA \otimes_k \cA')$ be the canonical inclusions.
We have
$$
\exp(1[-\b]) \tstar \exp(1[-\b']) = \exp(\iota(1[-\b]))\star \exp(\iota'(1[-\b']))
=  \exp(1[-\b \otimes 1 - 1 \otimes \b']);
$$
the second equation holds since
$\iota(1[-\b])$ and $\iota'(1[-\b'])$ commute. Therefore, for elements $\a \in HH^{II}(\cA)$ and $\a' \in HH^{II}(\cA')$, using also the associativity of $\star$, we get
$$
\begin{aligned}
(\phi)_*(\a) \tstar (\phi')_*(\a') 
& =  (\exp(1[-\b]) \star \rho(\a)) \tstar  (\exp(1[-\b']) \star \rho'(\a') ) \\
& =  (\exp(1[-\b]) \tstar \exp(1[-\b'])) \star (\rho(\a)   \tstar \rho'(\a') ) \\
& =   \exp(1[-\b \otimes 1 - 1 \otimes \b'])\star (\rho \otimes \rho')(\a    \tstar \a')  \\
& =  (\phi \otimes \phi')_*(\a \tstar \a').
\end{aligned}
$$
\end{proof}

\section{Hochschild homology of matrix factorization categories}
\label{HHofmf}
Let $k$ be a field, and let $Q$ be an essentially smooth $k$-algebra. Fix $f \in Q$. 

\subsection{Matrix factorizations}
\label{mf}
The dg-category $mf(Q, f)$ of \emph{matrix factorizations of $f$ over $Q$} is defined as follows:
\begin{itemize}
\item Objects are pairs $(P, \d_P)$, where $P$ is a finitely generated $\Z/2$-graded projective $Q$-module, and $\d_P$ is an odd degree endomorphism of $P$ such that $\d_P^2 = f \id_P$. 
\item $\Hom_{mf(Q, f)}((P, \d_P), (P', \d_{P'}))$ is the $\Z/2$-graded complex $\Hom_Q(P, P')$ with differential $\partial$ given by
$$
\partial(\a) = \d_{P'} \a- (-1)^{|\a|} \a \d_P
$$
for $\a$ homogeneous. From now on, we will omit the subscript on $\Hom_{mf(Q, f)}(-, -)$.
\end{itemize}
We emphasize that $f$ is allowed to be 0. The \emph{homotopy category of $mf(Q, f)$}, denoted $[mf(Q, f)]$, is the $Q$-linear category with the same objects as $mf(Q, f)$ and morphisms given by $\Hom_{[mf(Q, f)]}(-, -) := H^0 \Hom(- , -)$.

Let $X, Y \in mf(Q, f)$, and let $\a_0, \a_1 \in \Hom(X, Y)$ be cocycles. We recall that $\a_0, \a_1$ are \emph{homotopic} if there is an odd degree $Q$-linear map $h : X \to Y$ such that
$$
h d_X + d_{Y}h = \a_0 - \a_1.
$$
This is just the usual notion of a homotopy between morphisms of a $\Z/2$-graded complex, adapted verbatim to the setting of matrix factorizations. An object $X \in mf(Q, f)$ is \emph{contractible} if $\id_X$ is null-homotopic. Morphisms in $mf(Q, f)$ that are cocycles are homotopic if and only if they are equal in $[mf(Q, f)]$. 

\begin{defn} Given $X \in mf(Q,f)$, the {\em support of $X$} is the set
$$
\supp(X) = \{ \fp \in \Spec(Q) \mid \text{$X_\fp$ is not a contractible object of $mf(Q_\fp, f)$} \}.
$$
For a closed subset $Z$ of $\Spec(Q)$, let $mf^Z(Q,f)$ denote the full dg-subcategory of $mf(Q,f)$ consisting of those $X$ with $\supp(X) \subseteq Z$.
\end{defn}
We record the following:

\begin{prop} \label{prop116a}
Let $X \in mf(Q,f)$. 
\begin{enumerate}

\item When $f = 0$, $\supp(X)$ is the set of points at which the $\Z/2$-complex $X$ is not exact. Therefore, when $f = 0$, the notion of support defined above agrees with the usual notion of support for a $\Z/2$-graded complex. 

\item We have $\supp(X) \subseteq \Spec(Q/f)$. When $f$ is a non-zero-divisor, $\supp(X) \subseteq \Sing(Q/f)$. 

\end{enumerate}
\end{prop}

\begin{proof}

(1) This is \cite[Lemma 2.3]{BMTW}. (2) It is easy to check that any matrix factorization of a unit is contractible. Suppose $f$ is a non-zero-divisor. By
\cite[Theorem 3.9]{orlov}, the homotopy category $[mf(Q, f)]$ is equivalent to the singularity category of $Q/f$, and the singularity category is trivial when
$Q/f$ is regular. 
\end{proof}

\begin{rem} If $f$ is a non-zero-divisor, so that the morphism of schemes
$f: \Spec(Q) \to \A^1_k$ is flat, then
$$
\Spec(Q/f) \cap \Sing(f) = \Sing(Q/f),
$$ where $\Sing(f)$ denotes the set of points of 
  $\Spec(Q)$ at which the morphism $f: \Spec(Q) \to \A^1_k$ is not smooth. 
\end{rem}


Let $R$ be another essentially smooth $k$-algebra, and let $g \in R$. Given $X \in mf(Q, f)$ and $Y \in mf(R, g)$, we form the tensor product
$$
X \otimes Y \in mf(Q \otimes_k R, f \otimes 1 + 1 \otimes g)
$$
by adapting the notion of tensor product of $\Z/2$-graded complexes to matrix factorizations. The tensor product gives a dg-functor
$$
mf(Q, f) \otimes_k mf(R, g) \to mf(Q \otimes_k R , f \otimes 1 + 1 \otimes g).
$$
If $Z$ and $W$ are closed subsets of $\Spec(Q)$ and $\Spec(R)$, respectively, one has an induced functor
$$
mf^Z(Q, f) \otimes_k mf^W(R, g) \to mf^{Z \times W}(Q \otimes_k R , f \otimes 1 + 1 \otimes g).
$$
If $Q = R$, composing with multiplication in $Q$ gives a functor
$$
mf^Z(Q, f) \otimes_k mf^W(Q, g) \to mf^{Z \cap W}(Q , f + g).
$$

We also have a duality functor $D$ which determines an isomorphism of dg-categories
$$
D: mf(Q,f)^\op \xra{\cong} mf(Q, -f).
$$
The functor $D$ sends an object $P = (P,\d_P)$ of $mf(Q,f)$ to the object $P^* = (P^*, -\d_P^*)$ of $mf(Q,-f)$, and it sends an element $\a$ of 
$\Hom(P_2, P_1)^\op = \Hom(P_1, P_2)$
to the element $\a^*$ of $\Hom(P_2^*, P_1^*)$. Note that $\a^*(\gamma) = (-1)^{|\a||\gamma|} \gamma \circ \a$. 
If $X \in mf^Z(X, f)^{\op}$ for some closed $Z \subseteq \Spec(Q)$, then $D(X) \in mf^Z(X, -f)$. If $X, Y \in mf(Q, f)$, there is a canonical isomorphism
$$
\Hom(X, Y) \cong D(X) \otimes Y.
$$
In particular, if $X \in mf^Z(Q, f)$ and $Y \in mf^W(Q, f)$, we have
\begin{equation}
\label{homeqn}
\Hom(X, Y) \in mf^{Z \cap W}(Q, 0).
\end{equation}

\subsection{The HKR map} 
\label{sec:HKR}

Assume for the rest of Section \ref{HHofmf} that $\on{char}(k) = 0$. Given a $\Z$-graded complex $(C^\bullet, d)$ of $k$-vector spaces, its \emph{$\Z/2$-folding} is the $\Z/2$-graded complex whose even (resp. odd) component is  $\bigoplus_{i \in \Z} C^{2i}$ (resp. $\bigoplus_{i \in \Z} C^{2i + 1}$) and whose differential is given by $d$. 

Let $\Omega^\bu_{Q/k}$ denote the $\Z/2$-graded commutative $Q$-algebra given by the $\Z/2$-folding of the
exterior algebra over $\Omega^1_{Q/k}$. That is, $\Omega^\ev_{Q/k} = \bigoplus_j \Omega^{2j}_{Q/k}$, and
$\Omega^\odd_{Q/k} = \bigoplus_j \Omega^{2j+1}_{Q/k}$.
We write $(\Omega_{Q/k}^\bu, -df)$  for the $\Z/2$-graded complex of $Q$-modules with
underlying graded $Q$-module
$\Omega^\bu_{Q/k}$ and with differential given by left multiplication by  $-df \in \Omega^1_{Q/k}$. 

Let $Z$ be a closed subset of $\Spec(Q/f)$.  The goal of the rest of this section is to study,
for each triple $(Q, f, Z)$, a Hochschild-Kostant-Rosenberg (HKR)-type map
\begin{equation}
\label{HKRintro}
\e_{Q,f,Z}: HH(mf^Z(Q,f)) \to \R\Gamma_Z (\Omega_{Q/k}^\bu, -df).
\end{equation}
Here, $\R\Gamma_Z$ is the right adjoint of the inclusion functor $D^Z_{\Z/2}(Q) \subseteq D_{\Z/2}(Q)$,
where $D_{\Z/2}(Q)$ denotes the derived category of $\Z/2$-graded $Q$-modules, and $D^Z_{\Z/2}(Q) \subseteq D_{\Z/2}(Q)$ the subcategory spanned by complexes
with support contained in $Z$. 
It will be convenient for us to use the following \v{C}ech model for $\R \Gamma_Z$. Choose $g_1, \dots, g_m \in Q$ such that $Z = V(g_1, \dots, g_m)$, and let
$$
\cC = \cC(g_1, \dots, g_m) = \bigotimes_j (Q \to Q[1/g_i])
$$
be the ($\Z/2$-folding of the) augmented \v{C}ech complex. It is well-known that $\cC \otimes_Q M$ models $\R\Gamma_Z(M)$ for any $M \in D_{\Z/2}(Q)$; i.e., the functor 
$$
\cC \otimes_Q - : D_{\Z/2}(Q) \to D^Z_{\Z/2}(Q)
$$
is right adjoint to the inclusion. From now on, given $g_1, \dots, g_m \in Q$ such that $V(g_1, \dots, g_m) = Z$, we will tacitly identify $\R\Gamma_Z(M)$ with $\cC \otimes_Q M$. Note that, for any $\Z/2$-graded complex $M$ of $Q$-modules that is supported in $Z$, the natural morphism of complexes
\begin{equation}
\label{cechaugmentation}
\cC \otimes_Q M \to M
\end{equation}
given by the tensor product of the augmentation map $\cC \to Q$ with $\id_M$ is a quasi-isomorphism. 

HKR maps for matrix factorization categories have been widely studied. Segal and C\u ald\u araru-Tu give such an HKR map, involving Hochschild homology of the second kind and without
a support condition, in \cite[Corollary 3.4]{segal} and \cite[Theorem 4.2]{CT}, respectively; Efimov generalizes this result to the non-affine setting in \cite[Proposition 3.21]{efimov}; and Preygel
gives a map just as in (\ref{HKRintro}) (but also in the not-necessarily-affine setting), and proves it is a quasi-isomorphism, in \cite[Theorem
8.2.6(iv))]{preygel}. But \cite{preygel} doesn't contain a concrete formula for where the HKR map (\ref{HKRintro}) sends an element of the bar complex computing
$HH(mf^Z(Q,f))$, and we will need such a formula later on. So, we develop our own version of (\ref{HKRintro}). 

\subsubsection{Quasi-matrix factorizations} 
Define a curved dg-category $qmf(Q,f)$, the category of \emph{quasi-matrix factorizations}, in the following way. 
\begin{itemize}
\item Objects $(P, \d_P)$ are defined in the same way as those of $mf(Q, f)$, except we remove the requirement that $\d_P^2$ is given by multiplication by $f$. 
\item Morphisms are defined in the same way as in $mf(Q, f)$. 
\item The curvature element of $\End_{qmf(Q, f)}(P, \d_P)$ is $\d_P^2 - f$. 
\end{itemize}
$mf(Q, f)$ is precisely the full subcategory of $qmf(Q,f)$ spanned by objects with trivial curvature. Let $qmf(Q,f)^0$ denote the full subcategory of $qmf(Q,f)$ spanned by those objects $(P, \d_P)$ such that $\d_P = 0$. Note that the curvature element of
an object in $qmf(Q,f)^0$ is $-f$.  The pair $(Q,0)$ determines an object of $qmf(Q,f)^0$, and its endomorphisms form the curved differential $\Z/2$-graded algebra $(Q, 0, -f)$. That is, we have inclusions
$$
mf(Q, f) \into qmf(Q, f) \hookleftarrow qmf(Q, f)^0 \hookleftarrow (Q, 0, -f).
$$
These functors are all \emph{pseudo-equivalences}, in the language of \cite[Section 1.5]{PP}, and so, by \cite[Lemma A, page 5319]{PP}, the induced maps 
$$
HH^{II}(mf(Q, f)) \to HH^{II}(qmf(Q, f)) \leftarrow HH^{II}(qmf(Q, f)^0 ) \leftarrow HH^{II}(Q, 0, -f)
$$
are all quasi-isomorphisms.

A key point is that 
there is a (non-strict) cdg-functor
$$
(F, \beta): qmf(Q,f) \to qmf(Q,f)^0
$$
given by $F(P, \d_P) = (P,0)$ and  $\beta_{(P,\d_P)} = \d_P$.
The induced map
$$
(F,\beta)_*: HH^{II}(qmf(Q,f)) \to HH^{II}(qmf(Q,f)^0)
$$
sends $\a_0[\a_1| \cdots | \a_n]$, where $\a_i \in \Hom((P_{i+1}, \delta_{i+1}),(P_{i}, \delta_{i}))$, to 
$$
\sum_{i_0, \dots, i_n \geq 0}  (-1)^{i_0 + \cdots + i_n} 
\a_0 [\overbrace{\d_1| \cdots | \d_1}^{i_0}| \a_1 |\overbrace{\d_2| \cdots | \d_2}^{i_1}| \cdots | \a_n | \overbrace{\d_0| \cdots | \d_0}^{i_n}].
$$

\subsubsection{The supertrace}

Given a $\Z/2$-graded finitely generated projective $Q$-module $P$, define the {\em supertrace} map
$$
\str: \End_Q(P) \to Q
$$
as the composition
$$
\End_Q(P) \cong P^* \otimes_Q P \xra{\gamma \otimes p \mapsto \gamma(p)} Q
$$
for homogeneous elements $\gamma$, $p$. Equivalently, for $\a \in \End_Q(P)$ we have
$$
\str(\a) = 
\begin{cases}
\tr(\a_0: P_0 \to P_0) - \tr(\a_1: P_1 \to P_1) , & \text{if $\a$ has degree $0$, and} \\
0 , & \text{if $\a$ has degree $1$.} \\
\end{cases}
$$
Here, $\tr$ is the classical trace of an endomorphism of a projective module. We extend $\str$ to
a map
$$
\End_{\Omega^\bu_{Q/k}}(P \otimes_Q \Omega^\bu_{Q/k}) \cong
\End_Q(P) \otimes_Q \Omega^\bu_{Q/k}  \xra{\str \otimes \id} \Omega^\bu_{Q/k},
$$
which we also write as $\str$. 

\subsubsection{The HKR map without supports}
\label{nosupport}
\begin{defn}
A \emph{connection} on an object $(P, \d_P) \in qmf(Q,f)$ is a $k$-linear map
$$
\nabla : P \to \Omega^1_{Q/k} \otimes_Q P
$$
of odd degree such that $\n(qp) = dq \otimes p + q\n(p)$, i.e. a \emph{superconnection}, in the language of \cite{quillen}. Notice that the definition does not involve $\d_P$. 
\end{defn}

Choose a connection $\n_P$ on each object $(P,0) \in qmf(Q,f)^0$; we stipulate that the connection chosen for $Q \in qmf(Q,f)^0$ is the canonical one
given by the de Rham differential, $d: Q \to \Omega^1_{Q/k}$. Define
$$
\e^0: HH^{II}(qmf(Q,f)^0)^\nat \xra{} \Omega_{Q/k}^\bu
$$
by
$$
\e^0(\a_0[\a_1| \cdots |\a_m]) = \frac{1}{m!} \str(\a_0 \a'_1 \cdots \a'_m),
$$
where, for $\a: (P_1, 0) \to (P_2, 0)$, we set $\a' = \n_{P_2} \circ \a- (-1)^{|\a|} \a \circ \n_{P_1}$. By \cite[Theorem 5.18]{BW}, $\e^0$ gives a chain map
$$
HH^{II}(qmf(Q,f)^0) \xra{} (\Omega_{Q/k}^\bu, -df).
$$
Then the composition
$$
\e^Q: HH^{II}(Q,0, -f) \xra{\simeq} HH^{II}(qmf(Q,f)^0) \xra{\e^0}  (\Omega_{Q/k}^\bu, -df),
$$
where the first map is induced by inclusion, is given by the classical HKR map
$$
\e^Q(q_0 [q_1 | \cdots |q_n]) = \frac{q_0 dq_1 \cdots dq_n}{n!} \in \Omega^n_{Q/k}.
$$
In particular, $\e^0$ is a quasi-isomorphism. $(F, \b)_*$ is also a quasi-isomorphism, since
$$
qmf(Q,f)^0 \xra{\simeq} qmf(Q,f) \xra{(F,\beta)} qmf(Q,f)^0
$$
is the identity.

We define the HKR map
$$
\e_{Q, f}: HH(mf(Q,f)) \to (\Omega_{Q/k}^\bu, -df)
$$
to be the composition
\begin{align*}
HH(mf(Q,f)) \xra{\can} HH^{II}(mf(Q,f)) \xra{\simeq}
HH^{II}(qmf(Q,f)) &  \xra{(F, \beta)_*} HH^{II}(qmf(Q,f)^0) \\
&  \xra{\e^0} (\Omega_{Q/k}^\bu, -df),
\end{align*}
where ``can" denotes the canonical map. A more explicit formula for $\e_{Q, f}$ is given as follows. Given objects $(P_0, \d_0), \dots, (P_n, \delta_n)$ of $mf(Q,f)$ and maps
$$
P_0 \xla{\a_0} P_1 \xla{\a_1} \cdots \xla{\a_{n-1}} P_n  \xla{\a_n} P_0,
$$
set $\nabla_i = \nabla_{P_i}$. Then 
$$
\e_{Q,f}(\a_0[\a_1 | \dots | \a_n]) 
=
\sum_{i_0, \dots, i_n \geq 0}  \frac{(-1)^{i_0 + \cdots + i_n}}{(n + i_0 + \cdots + i_n)!}
\str \left(\a_0 (\d_1')^{i_0} \a_1' \cdots (\d_{n}')^{i_{n-1}} \a_n' (\d_0')^{i_n} \right),
$$
where, just as above,
$$
\a_j' = \n_{j} \circ \a_j - (-1)^{|\a_j|} \a_j \circ \n_{j+1}
\, \text{(with $\n_{n+1} = \n_0$)},
$$
and
$$
\delta_j' =  [\n_j, \delta_i] = \n_j \circ \delta_j + \delta_j \circ \n_j.
$$

Note that the sum in this formula is finite, since $\Omega_{Q/k}^j = 0$ for $j > \dm(Q)$. 

Summarizing, we have a commutative diagram
\begin{equation}
\label{bigdiagram}
\xymatrix{
HH(mf(Q,f)) \ar[rrdd]^-{\e_{Q,f}} \ar[r] & HH^{II}(mf(Q,f)) \ar[r]^-{\simeq} & HH^{II}(qmf(Q,f)) \ar[d]_-\simeq^-{(F,\beta)_*} & \\
&& HH^{II}(qmf(Q,f)^0)\ar[d]^-{\e^0}_-\simeq & HH^{II}(Q, 0, -f) \ar[l]_-\simeq \ar[dl]^-{\e^Q}_-\simeq \\
&& (\Omega_{Q/k}, -df). \\
}
\end{equation}
Notice that this implies $\e_{Q,f}$ is independent, up to natural isomorphism in the derived category, of the choices of connections. In particular,
the map on homology induced by $\e_{Q,f}$ is independent of such choices.

We include the following result, although it will not be needed in this paper:
\begin{prop}
\label{HKRquasi}
If the only critical value of $f : \Spec(Q) \to \A^1$ is 0, $\e_{Q, f}$ is a quasi-isomorphism.
\end{prop}
\begin{proof}
By \cite[Section 4.8, Corollary A]{PP}, the canonical map 
$$
HH(mf(Q, f)) \to HH^{II}(mf(Q, f))
$$
is a quasi-isomorphism. The statement therefore follows from the commutativity of diagram (\ref{bigdiagram}).
\end{proof}

\subsubsection{The HKR map with supports} 
We now define the HKR map for a general closed subset $Z$ of $\Spec(Q)$. Composing
$\e_{Q,f}$ with the natural map induced by the inclusion $mf^Z(Q,f) \subseteq mf(Q,f)$ gives a map
\begin{equation}
\label{zag}
HH(mf^Z(Q,f)) \to (\Omega_{Q/k}^\bu, -df).
\end{equation}
By Proposition \ref{prop116a} (1) and (\ref{homeqn}),
if $X,Y \in mf^Z(Q,f)$, $\Hom(X,Y)$ is a complex of $Q$-modules whose support is contained
in $Z$. (When $f$ is a non-zero-divisor, this complex is in fact supported on $Z \cap \Sing(Q/f)$.)
It follows that each row of the bicomplex used to define 
$HH(mf^Z(Q,f))$ is supported on $Z$. Since $HH(mf^Z(Q,f)$ is the direct sum totalization of this bicomplex, we have that
$HH(mf^Z(Q,f))$ is supported on $Z$. Adjointness thus gives a canonical isomorphism
$$
\e_{Q,f,Z}: HH(mf^Z(Q,f)) \to \R\Gamma_Z (\Omega_{Q/k}^\bu, -df)
$$
in $D(Q)$. 
In other words, $\e_{Q,f,Z}$ is represented in $D(Q)$ by the diagram
$$
\xymatrix{
HH(mf^Z(Q,f)) & \ar[l]^-{\simeq}_-{(\ref{cechaugmentation})} \R\G_Z HH(mf^Z(Q,f)) \ar[r]^-{(\ref{zag})}
&\R\Gamma_Z (\Omega_{Q/k}^\bu, -df).
}
$$
We will sometimes refer to $\e_{Q,f, Z}$ as just $\e$, if no confusion can arise. 

\subsection{Relationship between the HKR map and the map $I_f(0)$}

When $Q = \C[x_1, \dots, x_n]$ and $\fm = (x_1, \dots, x_n)$ 
is the only singular point of the map $f: \A^n_\C \to \A^1_\C$, Shklyarov defines in \cite[Section 4.1]{Shklyarov:higherresidues}
an isomorphism
$$
I_f(0): HH_*(mf(Q,f)) \xra{\cong} H_*(\Omega_{Q/k}^\bu, -df)
$$
as follows. Let $\cA_f$ be the endomorphism dga of the following matrix factorization $(P, \d_P)$ which represents the residue field $Q/\fm$ in the singularity category of $Q/f$: choose 
polynomials $y_1, \dots,
y_n \in Q$ so that $f = \sum_i x_i y_i$, let $P$ be the $\Z/2$-graded exterior algebra over $Q$ on generators $e_1, \dots, e_n$, and define a differential on $P$ given by
$$
\d_P = \sum_i
x_i e_i^* + y_i e_i.
$$
Here, $e_i^*$ is the $Q$-linear derivation of $P$ determined by $e_i^*(e_j) = \d_{ij}$.
By a theorem of Dyckerhoff (\cite[Theorem 5.2 (3)]{dyckerhoff}), the inclusion
$$
\iota: \cA_f \into mf(Q, f)
$$
is a Morita equivalence. Since Hochschild homology is Morita invariant, the induced map
$$
\iota_*: HH_*(\cA_f) \xra{\cong} HH_*(mf(Q,f))
$$
is an isomorphism. 

From now on, we identify $P$ with $Q \otimes_\C \L$, where $\L = \Lambda_\C(e_1, \dots, e_n)$, and $\cA_f$ with $Q \otimes_\C \End_\C(\L)$. Shklyarov defines a quasi-isomorphism
$$
\alpha: HH(\cA_f) \xra{\simeq} (\Omega_{Q/k}^\bu, -df)
$$
as the composition
$$
 HH(\cA_f) \xra{\exp(-1[\d_P])}  HH^{II}(\cA_f) \xra{\e'} (\Omega_{Q/k}^\bu, -df),
$$
where 
$$
\e'(q_0 \otimes \a_0 [q_1 \otimes \a_1| \cdots |q_n \otimes \a_n])
= \frac{(-1)^{\sum_{\text{ $i$ odd}} |\a_i|}}{n!} \str(\a_0 \cdots \a_n)q_0 dq_1 \cdots dq_n.
$$
Finally, $I_f(0)$ is the composition
$$
HH_*(mf(Q,f)) \xra{\iota_*^{-1}}
HH_*(\cA_f) \xra{\alpha}
H_*(\Omega_{Q/k}^\bu, -df).
$$

\begin{lem}
\label{If}
 The map $\e'$ coincides with the map $\e_{Q,f}$ restricted to $HH(\End(P))$ for the choice of connection $\n_P$ defined as
$\n_P(q \otimes \a) = dq \otimes \a$.
Thus, $I_f(0) = \e_{Q,f}$.
\end{lem}

\begin{proof} We have
$$
\begin{aligned}
\e_{Q,f}(q_0 \otimes \a_0 [q_1 \otimes \a_1| \cdots |q_n \otimes \a_n])
& = \frac{1}{n!} \str((q_0 \otimes \a_0) (dq_1 \otimes \a_1) \cdots (dq_n \otimes \a_n)) \\
& = \frac{(-1)^{\sum_i i |\a_i| }}{n!} \str(\a_0 \cdots \a_n) q_0 dq_1 \cdots dq_n \\
& = \frac{(-1)^{\sum_{\text{ $i$ odd}} |\a_i|}}{n!} \str(\a_0 \cdots \a_n)q_0 dq_1 \cdots dq_n. \\
\end{aligned}
$$
\end{proof}

\subsection{Compatibility of the HKR map with taking duals}
\label{duals}
Shklyarov proves in \cite[Proposition 3.2]{Shklyarov:Toward} that, for any differential $\Z/2$-graded algebra $\cA$, there is a canonical isomorphism of complexes
\begin{equation} \label{E1215c}
\Phi: HH(\cA) \xra{\cong} HH(\cA^\op)
\end{equation}
given by
$$
a_0[a_1| \cdots |a_n] \mapsto (-1)^{n + \sum_{1 \leq i < j \leq n} (|a_i| - 1)(|a_j| -1)} a^\op_0[a^\op_n| \cdots | a^\op_1],
$$
where, for $a \in \cA$, $a^\op$ denotes $a$ regarded as an element of $\cA^{\op}$. The same formula gives an isomorphism
$$
HH(\cC) \xra{\cong} HH(\cC^\op)
$$
for any curved differential $\G$-graded category $\cC$, where $\G \in \{\Z, \Z/2 \}$.

Composing $\Phi$ and $D$, where $D$ is the dualization functor defined in Section \ref{mf}, we obtain the isomorphism of complexes
\begin{equation} \label{E1215d}
\Psi: HH(mf^Z(Q,f)) \xra{\cong}  HH(mf^Z(Q,-f))
\end{equation}
given explicitly by
$$
\Psi(a_0[a_1| \cdots |a_n]) = 
(-1)^{n + \sum_{1 \leq i < j \leq n} (|a_i| - 1)(|a_j| -1)} a^*_0[a^*_n| \cdots | a^*_1].
$$

\begin{lem}\label{lem1129}
  The diagram
$$
\xymatrix{
HH(mf^Z(Q,f)) \ar[d]^-{\Psi} \ar[r]^-{\e_{Q,f,Z}} & \R\Gamma_Z (\Omega_{Q/k}^\bu, -df) \ar[d]^-\gamma \\
HH(mf^Z(Q,- f)) \ar[r]^-{\e_{Q,-f, Z}} & \R\Gamma_Z  (\Omega_{Q/k}^\bu, df) \\
}
$$
commutes in $D(Q)$, where $\gamma$ is $\R\Gamma_Z$ applied to the map whose restriction to $\Omega^j_{Q/k}$ is multiplication by $(-1)^j$ for all $j$. 
\end{lem}

\begin{proof} 
The map $\e_{Q, f, Z}$ factors as
$$
HH(mf^Z(Q,f)) \to \R\Gamma_Z HH(mf(Q,f)) \xra{\e_{Q, f}} (\Omega_{Q/k}^\bu, -df),
$$
where the first map is the canonical one. $\e_{Q,- f, Z}$ factors similarly. Since the diagram
$$
\xymatrix{
HH(mf^Z(Q,f)) \ar[d]^-{\Psi} \ar[r] & \R\Gamma_Z HH(mf(Q,f)) \ar[d]^-{\R \G_Z (\Psi)} \\
HH(mf^Z(Q,- f)) \ar[r] & \R\Gamma_Z  HH(mf(Q,-f)) \\
}
$$
evidently commutes, we may assume $Z = \Spec(Q)$. 

Recall from (\ref{bigdiagram}) that $\e_{Q, f}$ fits into a commutative diagram
\begin{equation} \label{E1119b}
\xymatrix{
HH(mf(Q,f)) \ar[r]^-\theta \ar[dr]_-{\e_{Q,f}} & HH^{II}(qmf(Q,f)^0) \ar[d]^-{\e^0}_-\simeq & HH^{II}(Q, 0, -f) \ar[l]_-\can^-\simeq \ar[dl]^-{\e^Q}_-\simeq \\
& (\Omega^\bu_{Q/k}, -df), \\
}
\end{equation}
where 
$$
\theta(\a_0[\a_1| \cdots | \a_n]) = 
\sum_{i_0, \dots, i_n \geq 0}  (-1)^{i_0 + \cdots + i_n} 
\a_0 [\d_1^{i_0}| \a_1 |\d_2^{i_1}| \cdots | \a_n | \d_0^{i_n}].
$$
Here, $\d^i$ stands for $\overbrace{\d| \cdots | \d}^i$. 

The map $\Psi$ extends to a map 
$$
\Psi: HH^{II}(qmf(Q,f)^0) \to HH^{II}(qmf(Q, -f)^0)
$$ 
using the same formula, and this map in turn restricts to a map
$$
\Psi: HH^{II}(Q,0, -f) \to HH^{II}(Q,0, f)
$$
given by 
$$
\Psi(q_0[q_1 | \cdots |q_n]) = (-1)^{n + {n \choose 2}} q_0[q_n| \cdots | q_1].
$$
We claim that the diagram
\begin{equation} \label{E1128}
\xymatrix{
HH(mf(Q,f)) \ar[r]^-\theta \ar[d]^\Psi & HH^{II}(qmf(Q,f)^0) \ar[d]^-\Psi & HH^{II}(Q, 0, -f) \ar[l]_-\can^-\simeq \ar[d]^-\Psi \\
HH(mf(Q,-f)) \ar[r]^-\theta  & HH^{II}(qmf(Q,-f)^0) & HH^{II}(Q, 0, f) \ar[l]_-\can^-\simeq \\
}
\end{equation} 
commutes. This is evident for the right square. As for the left, the element $\a_0[\a_1| \cdots | \a_n]$ is mapped via $\Psi \circ \theta$ to 
$$
\sum_{i_0, \dots, i_n \geq 0}  (-1)^{I}  (-1)^{n + I + \sum_{1 \leq i < j \leq n} (|\a_i| - 1)(|\a_j| -1)} 
\a_0^* [(\d^*_0)^{i_n} | \a_n^* | \cdots | (\d_2^*)^{i_1} | \a^*_1 | (\d_1^*)^{i_0}],
$$
where $I = i_0 + \cdots + i_n$. The sign is correct since $|\d_i| -1$ is even for all $i$. The map $\theta \circ \Psi$ sends 
$\a_0[\a_1| \cdots |\a_n]$ to
$$
\sum_{j_0, \dots, j_n \geq 0}  (-1)^{J} 
(-1)^{n + \sum_{1 \leq i < j \leq n} (|\a_i| - 1)(|\a_j| -1)} \a^*_0[(-\d^*_0)^{j_0}| \a^*_n| 
\cdots | (-\d_2^*)^{j_{n-1}} | \a^*_1 | (-\d_1^{j_n})^*],
$$
where $J = j_0 + \cdots + j_n$. The reason for the minus sign in $(-\d_j^*)^i$ is that the differential of $(P, d)^*$ is $-d^*$. Since these two expressions are equal, the left square commutes.

Using the commutativity of the diagrams  \eqref{E1119b} and \eqref{E1128}, it suffices to prove that the square
$$
\xymatrix{
HH^{II}(Q,0,-f) \ar[r]^\Psi \ar[d]^{\e^Q} & HH^{II}(Q,0,f) \ar[d]^{\e^Q}  \\
(\Omega^\bu_{Q/k}, -df) \ar[r]^\gamma &(\Omega^\bu_{Q/k}, df) \\
}
$$
commutes. This holds since $\Omega^\bu_{Q/k}$ is graded commutative, so that
$$
\begin{aligned}
\gamma \e^Q(q_0[q_1 | \cdots |q_n]) & = \frac{(-1)^n}{n!} q_0 dq_1 \cdots dq_n \\
& = \frac{(-1)^{n + {n \choose 2}}}{n!} q_0 dq_n \cdots dq_1 \\
& = \e^Q \Psi(q_0[q_1 | \cdots |q_n]). \\
\end{aligned}
$$
\end{proof}

\subsection{Multiplicativity of the HKR map}
Let $(Q,f, Z)$ and $(R, g, W)$ be triples consisting of an essentially smooth $k$-algebra, an element of the algebra, and a closed subset of the spectrum of
the algebra. The tensor product of matrix factorizations (Section \ref{mf}), along with the K\"unneth map for Hochschild homology of dg-categories (Section \ref{kunnethcdgc}), gives a pairing
\begin{equation} \label{E1215a}
- \tstar - : HH(mf^Z(Q,f)) \otimes_k HH(mf^{W}(R,g)) \to HH(mf^{Z \times W}(Q \otimes_k R, f \otimes 1 + 1 \otimes g)).
\end{equation}

Write $f+g$ for the element $f \otimes 1 + 1 \otimes g \in Q \otimes_k R$. Multiplication in $\Omega^\bu_{Q \otimes_k R/k}$ defines a pairing of complexes of $Q
\otimes_k R$-modules
$$
- \smsh - : (\Omega_{Q/k}^\bu, -df) \otimes_k (\Omega^\bu_{R/k}, -dg)\to (\Omega^\bu_{Q \otimes_k R/k}, -df - dg).
$$
We compose this with the canonical maps $\R \Gamma_Z (\Omega_{Q/k}^\bu, -df) \to (\Omega_{Q/k}^\bu, -df)$ and
$\R \Gamma_W (\Omega^\bu_{R/k}, -dg) \to (\Omega^\bu_{R/k}, -dg)$ to obtain the map
$$
\R\Gamma_Z (\Omega_{Q/k}^\bu, -df) \otimes_k \R \Gamma_W (\Omega^\bu_{R/k}, -dg) \to (\Omega^\bu_{Q \otimes_k R/k}, -df - dg).
$$
The source of this map is supported on the closed subset $Z \times W$ of $\Spec(Q \otimes_k R) = \Spec(Q) \times_k \Spec(R)$. Thus, by adjointness, we obtain a pairing
\begin{equation} \label{E1215b}
- \smsh -: \R\Gamma_Z (\Omega_{Q/k}^\bu, -df) \otimes_Q \R \Gamma_W (\Omega^\bu_{R/k}, -dg) \to \R\G_{Z \times  W} (\Omega^\bu_{Q \otimes_k R/k},-df-dg).
\end{equation}

A key fact is that the pairings \eqref{E1215a} and \eqref{E1215b} are compatible via the HKR maps:

\begin{prop} \label{prop112}
  The diagram
$$
\xymatrix{
  HH(mf^Z(Q,f)) \otimes_k  HH(mf^W(R,g)) \ar[rr]^-{\e_{Q, f, Z} \otimes \e_{R, f, W}} \ar[d]_-{- \tstar -}
  & &\R\G_Z (\Omega_{Q/k}^\bu, -df)  \otimes_k \R\G_W (\Omega^\bu_{R/k}, -dg) \ar[d]^-{\smsh} \\
  HH(mf^{Z \times W}(Q \otimes_k R, f+g))  \ar[rr]^-{\e_{Q \otimes_k R, f + g, Z \times W}} && \R\Gamma_{Z \times W} (\Omega^\bu_{Q \otimes_k R/k}, -df -dg)  \\
}
$$
in $D(Q \otimes_k R)$ commutes.
\end{prop}

\begin{proof} 
It is enough to show the diagrams
\begin{equation}
\label{dia1}
\xymatrix{
  HH(mf^Z(Q,f)) \otimes_k  HH(mf^W(R,g)) \ar[r] \ar[d]_-{- \tstar -}
  & \R\G_Z  HH(mf(Q,f)) \otimes_k  \R\G_W HH(mf(R,g)) \ar[d]_-{- \tstar -}\\
  HH(mf^{Z \times W}(Q \otimes_k R, f+g))  \ar[r] & \R\Gamma_{Z \times W} HH(mf(Q \otimes_k R, f+g))  \\
}
\end{equation}
and 
\begin{equation}
\label{dia2}
\xymatrix{
  \R\G_Z  HH(mf(Q,f)) \otimes_k  \R\G_W HH(mf(R,g)) \ar[d]_-{- \tstar -} \ar[r]
\ar[d]_-{- \tstar -}
  & \R\G_Z (\Omega_{Q/k}^\bu, -df)  \otimes_k \R\G_W (\Omega^\bu_{R/k}, -dg) \ar[d]^-{\smsh} \\
  \R\Gamma_{Z \times W} HH(mf(Q \otimes_k R, f+g))   \ar[r]
& \R\Gamma_{Z \times W} (\Omega^\bu_{Q \otimes_k R/k}, -df -dg)  \\
}
\end{equation} 
commute. Here, the right-most vertical map in (\ref{dia1}) (which coincides with the left-most vertical map in (\ref{dia2})) is defined in a manner similar to the map (\ref{E1215b}), and the horizontal maps in (\ref{dia1}) are the canonical ones. The commutativity of (\ref{dia1}) is clear. As for (\ref{dia2}), it suffices to show the diagram
$$
\xymatrix{
 HH(mf(Q,f)) \otimes_k  HH(mf(R,g)) \ar[d]_-{- \tstar -} \ar[rrr]^-{\e_{Q,f} \otimes \e_{R, g}} \ar[d]_-{- \tstar -}
  &&&(\Omega_{Q/k}^\bu, -df)  \otimes_k (\Omega^\bu_{R/k}, -dg) \ar[d]^-{\smsh} \\
HH(mf(Q \otimes_k R, f+g))   \ar[rrr]^-{\e_{Q \otimes_k R, f + g}} &&&  (\Omega^\bu_{Q \otimes_k R/k}, -df -dg)  \\
}
$$
in $D(Q \otimes_k R)$ commutes. Factoring the HKR maps as in diagram (\ref{bigdiagram}), it suffices to show the squares
\begin{equation}
\label{dia3}
\xymatrix{
HH(mf(Q,f)) \otimes_k  HH(mf(R,g)) \ar[d]_-{- \tstar -} \ar[r] & HH^{II}(qmf^0(Q,f)) \otimes_k  HH^{II}(qmf^0(R,g))\ar[d]_-{- \tstar -} \\
HH(mf(Q \otimes_k R, f+g)) \ar[r]  & HH^{II}(qmf^0(Q \otimes_k R, f+g))
}
\end{equation}
and 
\begin{equation}
\label{dia4}
\xymatrix{
HH^{II}(qmf(Q,f)) \otimes_k  HH^{II}(qmf(R,g)) \ar[r]^-{\e^0 \otimes \e^0} \ar[d]_-{- \tstar -} & (\Omega_{Q/k}^\bu, -df)  \otimes_k (\Omega^\bu_{R/k}, -dg) \ar[d]^-{\smsh}  \\
HH^{II}(qmf(Q \otimes_k R, f+g)) \ar[r]^-{\e^0} & (\Omega^\bu_{Q \otimes_k R/k}, -df -dg) 
}
\end{equation}
commute. It follows immediately from Lemma \ref{lem112} that (\ref{dia3}) commutes. The square 
\begin{equation}
\label{dia5}
\xymatrix{
HH^{II}(Q,-f) \otimes_k  HH^{II}(R, -g) \ar[r]^-{ \simeq} \ar[d]^-{- \tstar -}& HH^{II}(qmf(Q,f)) \otimes_k  HH^{II}(qmf(R,g))  \ar[d]_-{- \tstar -}  \\
HH^{II}(Q \otimes_k R, -f-g) \ar[r]^-{\simeq}& HH^{II}(qmf(Q \otimes_k R, f+g))
}
\end{equation}
evidently commutes, and concatenating this diagram with (\ref{dia4}) gives a commutative diagram. It follows that (\ref{dia4}) commutes.
\end{proof}


For an essentially smooth $k$-algebra $Q$, any element $f \in Q$, and any pair of closed subsets $Z$ and $W$ of $\Spec(Q)$,  there is a pairing
\begin{equation} \label{E1116}
HH(mf^Z(Q, f)) \times HH(mf^W(Q, -f)) \xra{\star} 
HH(mf^{Z \cap W}(Q,0))
\end{equation}
defined by composing the K\"unneth map
$$
HH(mf^Z(Q, f)) \times HH(mf^W(Q, -f)) \xra{\tstar} 
HH(mf^{Z \times W}(Q \otimes_k Q, f \otimes 1 - 1 \otimes f)) 
$$
with the map
$$
HH(mf^{Z \times W}(Q \otimes_k Q, f \otimes 1 - 1 \otimes f)) \to 
HH(mf^{Z \cap W}(Q,0))
$$
induced by the multiplication map $Q \otimes Q \to Q$. The previous result, along with the functoriality of the HKR map, yields:

\begin{cor} \label{cor1129}
  The diagram
$$
\xymatrix{
  HH(mf^Z(Q,f)) \otimes_k  HH(mf^W(Q,-f)) \ar[rrr]^-{\e_{Q,f,Z} \otimes \e_{Q, -f, Z}} \ar[d]
  &&& \R\G_Z (\Omega_{Q/k}^\bu, -df) \otimes_k \R\G_W (\Omega_{Q/k}^\bu, df)  \ar[d]^-{\smsh} \\
  HH(mf^{Z \cap W}(Q, 0))  \ar[rrr]^-{\e_{Q, 0, Z \cap W}} &&& \R\Gamma_{Z \cap W} \Omega^{\bu}_{Q/k} \\
}
$$
in $D(Q \otimes_k Q)$ commutes.
\end{cor}

We will be especially interested in the case where $Z \cap W = \{\fm\}$.

\section{Proof of Shklyarov's conjecture}
\label{traceresidue}

Throughout this section, we assume
\begin{itemize}
\item $k$ is a field,
  \item $Q$ is a regular $k$-algebra, and
\item $\fm$ is a $k$-rational maximal ideal of $Q$; i.e.~the canonical map $k \to Q/\fm$ is an isomorphism.
\end{itemize}

Let us review our progress on the proof of Conjecture \ref{conj:s}.
Recall from the introduction that, to prove the conjecture, it suffices to show that diagram \eqref{outline} commutes, the composition along
the left side of this diagram computes the
pairing $\eta_{mf}$, and the composition along the right side computes the residue pairing.
So far, we have shown the two interior squares of \eqref{outline} commute: this follows from
Lemma \ref{If}, Lemma \ref{lem1129}, and Corollary \ref{cor1129}. In this section, we show the left side of the diagram gives the canonical pairing $\eta_{mf}$ (Lemma \ref{lem1119t}), the right side of the diagram gives the residue pairing (Proposition \ref{prop1129}), and the bottom triangle commutes (Theorem \ref{thm112}).

\subsection{Computing $HH(mf^\fm(Q,0))$}
\label{calculation}
We carry out a calculation of the Hochschild homology of the dg-category $mf^\fm(Q, 0)$ that we will use repeatedly throughout the rest of the paper. Let $n$ denote the Krull dimension of $Q_\fm$. We recall that a sequence $x_1, \dots, x_n \in \fm$ is called a \emph{system of parameters} if $x_1, \dots, x_n$ generate an $\fm$-primary ideal, and a system of parameters is called \emph{regular} if the elements generate $\fm$. 

Fix a regular system of parameters $x_1, \dots, x_n$ for $Q_\fm$, and set $K = \Kos_{Q_\fm}(x_1, \dots, x_n) \in mf^\fm(Q_\fm, 0)$, the $\Z/2$-folded Koszul
complex on the $x_i$'s.  
Explicitly, $K$ is the differential $\Z/2$-graded algebra whose underlying algebra is the exterior algebra over $Q_\fm$ generated by $e_1, \dots, e_n$ with
$d^K(e_i) = x_i$. 
The differential $\Z/2$-graded $Q_\fm$-algebra $\cE:= \End_{mf^\fm(Q_\fm, 0)}(K)$ is generated by odd degree elements $e_1, \dots, e_n$, $e^*_1, \dots, e^*_n$ satisfying
$e_i^2 = 0 = (e_i^*)^2$, $[e_i, e_j] = 0 = [e_i^*, e_j^*]$, 
and $[e_i, e_j^*] = \d_{ij}$; and the differential $d^\cE$ is determined by the equations $d^\cE(e_i) = x_i$ and $d^\cE(e_i^*) =
0$. Let 
$\L$ be the dg-$k$-subalgebra of $\cE$ generated by the $e_i^*$. So, $\L$ is an exterior algebra over $k$ on $n$ generators, with trivial
differential. The inclusion $\L \subseteq \cE$ is a quasi-isomorphism of differential $\Z/2$-graded $k$-algebras. Since $\L$ is graded commutative, $HH_*(\L)$ is a $k$-algebra under the shuffle product, and, by a standard calculation, there is an isomorphism
\begin{equation}
\label{HHexterior}
\Lambda \otimes_k k[y_1, \dots, y_n] \xra{\cong} HH_*(\L),
\end{equation}
of $k$-algebras, where $e_i^*  \otimes 1 \mapsto e_i^*[]$, and $1 \otimes y_i \mapsto 1[e_i^*]$. 
Here, and throughout the paper, we use the notation $\a_0[]$ to denote an element of a Hochschild complex of the form $\a_0[\a_1 | \cdots | \a_n]$ with $n = 0$.

\begin{lem} \label{lem529}
The canonical morphisms 
\begin{equation}
\label{cE}
\cE \into mf^\fm(Q_\fm,0)
\end{equation}
and 
\begin{equation}
\label{loc}
mf^\fm(Q, 0) \to mf^\fm(Q_\fm,0)
\end{equation}
of dg-categories are Morita equivalences. In particular, we have canonical quasi-isomorphisms
\begin{equation}
\label{lambdaqi}
HH(\Lambda) \xra{\simeq} HH(mf^\fm(Q_\fm,0)) \xleftarrow{\simeq} HH(mf^\fm(Q, 0)).
\end{equation}
\end{lem}

\begin{proof} 
To prove (\ref{cE}) is a Morita equivalence, we prove the thick closure of $K$ in the homotopy category $[mf^\fm(Q_\fm,0)]$ is all of
  $[mf^\fm(Q_\fm,0)]$. Let $\cD$ denote
  the derived category of all $\Z/2$-complexes of finitely generated $Q_\fm$-modules whose homology groups are finite dimensional over $k$. Since $Q_\fm$ is regular, it follows from \cite[Proposition 3.4]{BMTW} that the canonical functor 
  $$
  [mf^\fm(Q_\fm, 0)] \to \cD
  $$
  is an equivalence. It therefore suffices to show $\Thick(K) = \cD$; in fact, we need only show every object in $\cD$ with free components is in $\Thick(K)$.

Let $X$ be an object of $\cD$ with free components. We may assume that $X$ is \emph{minimal}, i.e. that $k \otimes_{Q_\fm} X$ is a direct sum of copies of $k$ and $\Sigma k$. The isomorphism $K \xra{\cong} k$ in $\cD$ induces an isomorphism 
$$
K \otimes_{Q_\fm} X \xra{\cong} k \otimes_{Q_\fm} X,
$$
and therefore $K \otimes_{Q_\fm} X \in \Thick(k)$. It thus suffices to prove $X \in \Thick(K \otimes_{Q_\fm} X)$. Since
  $$
  K \otimes_{Q_\fm} X \cong  \Kos_{Q_\fm}(x_1) \otimes_{Q_\fm} \cdots \otimes_{Q_\fm} \Kos_{Q_\fm}(x_n)  \otimes_{Q_\fm} X,
  $$
it suffices to show that, for every $Y \in \cD$ whose components are free $Q_\fm$-modules, and every $x \in  \fm \setminus \{0\}$, $Y \in \Thick(Y/xY)$.
  Using induction and the exact sequence
  $$
  0 \to Y/x^{n-1}Y \xra{x} Y/x^nY \to Y/xY \to 0,
  $$ 
  we get $Y/x^n Y \in \Thick(Y/xY)$ for all $n$.
Observing that $\End_{\cD}(Y)[1/x] = 0$, choose $n \gg 0$ such that 
multiplication by $x^n$ on $Y$ determines the zero map in $\cD$. The distinguished triangle
$$
Y \xra{x^n} Y \to Y/x^n \to \Sigma Y
$$
in $\cD$ therefore splits, implying that $Y$ is a summand of $Y/x^n$. Thus, $ Y \in \Thick(Y/x^n) \subseteq \Thick(Y/xY)$.

As for (\ref{loc}), the functor $[mf^\fm(Q, 0)] \to [mf^\fm(Q_\fm, 0)]$ is fully faithful, since 
$\Hom_{[mf^\fm(Q,0)]}(X, Y)$ is supported in $\{\fm\}$ for any $X, Y$. It follows that the induced map
\begin{equation}
\label{idem}
[mf^\fm(Q, 0)]^{\on{idem}} \to [mf^\fm(Q_\fm, 0)]^{\on{idem}}
\end{equation}
on idempotent completions is fully faithful, so we need only show (\ref{idem}) is essentially surjective. By the above argument, it suffices to show $K$ is in the essential image of (\ref{idem}). Choose a $Q$-free resolution $F$ of $k$; $F_\fm$ is homotopy equivalent to the Koszul complex on the $x_i$'s, and so the $\Z/2$-folding of $F_\fm$ is isomorphic to $K$ in $[mf^\fm(Q_\fm, 0)]$.
\end{proof}

\begin{rem}
\label{moritaremark}
Let $\widehat{Q}$ denote the $\fm$-adic completion of $Q$. Letting $\widehat{Q}$ play the role of $Q$ in Lemma \ref{lem529} implies that the inclusion
$$
\on{End}_{mf^\fm(\widehat{Q}, 0)}(K \otimes_{Q_\fm} \widehat{Q}) \into mf^\fm(\widehat{Q}, 0)
$$
is a Morita equivalence. The same proof that shows the map (\ref{loc}) in Lemma \ref{lem529} is a Morita equivalence shows the canonical map
$$
mf^\fm(Q, 0) \to mf^\fm(\widehat{Q}, 0)
$$
is a Morita equivalence.
\end{rem}

\subsection{The trace map}
We define an even degree map
$$
\trace: HH_*(mf^\fm(Q,0)) \to k
$$
of $\Z/2$-graded $k$-vector spaces, with $k$ concentrated in even degree, as follows. Let $\Perf_{\Z/2}(k)$ denote the dg-category of
$\Z/2$-graded complexes of (not necessarily finitely dimensional) 
$k$-vector spaces having finite dimensional homology. There is a dg-functor  $mf^{\fm}(Q,0) \to \Perf_{\Z/2}(k)$ 
induced by restriction of scalars along the structural map $ k \to Q$ that induces a map 
$$
u: HH_*(mf^\fm(Q,0)) \to HH_*(\Perf_{\Z/2}(k)),
$$
and there is a canonical isomorphism
$$
v : k \xra{\cong} HH_*(\Perf_{\Z/2}(k)) 
$$
given by $a \mapsto a[]$. Here, $k$ is considered as a $\Z/2$-graded complex concentrated in even degree, and, on the right, $a$ is regarded as an endomorphism of this complex. We define
$$
\trace:= v^{-1} u.
$$
In the rest of this subsection, we establish several technical properties of the trace map that we will need later on.

Given an object  $(P, \d_P) \in mf^\fm(Q,0)$, there is a canonical map of complexes $\End(P) \to HH(mf^{\fm}(Q,0))$ given by $\a \mapsto \a[]$
and hence an induced map
\begin{equation}
\label{inducedmap}
H_*(\End(P))\to   HH_*(mf^\fm(Q,0)).
\end{equation}

\begin{prop}
\label{(1)}
If $(P, \d_P) \in mf^\fm(Q,0)$, and $\a$ is an even degree endomorphism of $P$, the composition
$$
H_0(\End(P)) \xra{(\ref{inducedmap})} HH_0(mf^\fm(Q, 0)) \xra{\trace} k
$$
sends $\a$ to the supertrace of the endomorphism of $H_*(P)$ induced by $\a$:
\begin{align*}
\trace(\a[]) & = \on{str}(H_*(\a) : H_*(P) \to H_*(P)) \\
& = \tr(H_0(\a) : H_0(P) \to H_0(P)) - \tr(H_1(\a) : H_1(P) \to H_1(P)).
\end{align*}
In particular,
$$
\trace(\id_P[]) = \dim_k H_0(P) - \dim_k H_1(P).
$$
\end{prop}

\begin{proof} 
Let $\on{Vect_{\Z/2}}(k)$ denote the subcategory of $\Perf_{\Z/2}(k)$ spanned by finite-dimensional $\Z/2$-graded vector spaces with trivial differential. It is well-known that the inclusion $\on{Vect_{\Z/2}}(k) \into \Perf_{\Z/2}(k)$ induces a quasi-isomorphism on Hochschild homology. Composing the map $\End(H_*(P)) \to HH_*(\on{Vect}_{\Z/2}(k))$ given by $\a \mapsto \a[]$ with the canonical map $H_*(\End(P)) \to \End(H_*(P))$ gives a map
\begin{equation}
\label{vs}
H_*(\End(P)) \to HH_*(\on{Vect}_{\Z/2}(k)).
\end{equation}

We first show that the square
\begin{equation}
\label{homologydia}
\xymatrix{
H_*(\End(P)) \ar[d]^-{(\ref{vs})} \ar[r]^-{(\ref{inducedmap})} & HH_*(mf^\fm(Q, 0)) \ar[d]^-{u} \\
HH_*(\on{Vect_{\Z/2}}(k)) \ar[r]^-{\cong}& HH_*(\Perf_{\Z/2}(k)) \\
}
\end{equation}
commutes. Let $\b$ be an even degree cycle in $\End(P)$, and let $H_*(\b)$ denote the induced endomorphism of $H_*(P)$. We must show the cycles $\b[]$ and $H_*(\b)[]$ coincide in $HH_*(\Perf_{\Z/2}(k))$. To see this, choose even degree $k$-linear chain maps
$$
\iota : H_*(P) \to P \text{, } \pi : P \to H_*(P)
$$
such that
\begin{itemize}
\item  $\pi \circ \iota = \id_{H_*(P)}$, and
\item $\iota \circ \pi$ is homotopic to $\id_P$ via a ($\Z/2$-graded) homotopy $h$, i.e.
$$
\iota \circ \pi - \id_P = \d_P \circ h + h \circ \d_P.
$$
\end{itemize}
Applying the Hochschild differential $b$ to
$$
\pi [\b \circ \iota] \in \Hom(P, H_*(P)) \otimes \Hom(H_*(P), P) \subseteq HH(\Perf_{\Z/2}(k)),
$$ 
we get
$$
b( \pi [\b \circ \iota]) = (b_2 + b_1)(\pi [\b \circ \iota]) = b_2(\pi [\b \circ \iota]) = (\pi \circ \b \circ \iota)[] - (\iota \circ \pi \circ \beta)[] = H_*(\b)[] - (\iota \circ \pi \circ \beta)[].
$$
Next, observe that 
$$
(b_2 + b_1)(( h \circ \b)[]) = b_1((h \circ \b)[]) = (\iota \circ \pi \circ \b - \b) [].
$$
It follows that diagram (\ref{homologydia}) commutes.

The isomorphism
$$
v: k \xra{\cong} HH_*(\Perf_{\Z/2}(k))
$$
factors as 
$$
k \xra{\cong} HH_*(k) \xra{\cong}  HH_*(\on{Vect_{\Z/2}}(k)) \xra{\cong}  HH_*(\Perf_{\Z/2}(k)),
$$
where each map is the evident canonical one. There is a chain map $HH( \on{Vect}_{\Z/2}(k)) \to HH(k)$ given by the
generalized trace map described in \cite[Section 2.3.1]{segal} and an evident isomorphism $HH_*(k) \xra{\cong} k$. It follows from \cite[Lemma 2.12]{segal} that
composing these maps gives the inverse of 
$$
k \xra{\cong} HH_*(k) \xra{\cong} HH_*(\on{Vect_{\Z/2}}(k)).
$$
As discussed in \cite[Page 872]{segal}, the generalized trace sends a class of the form $\a_0[]$ to $\str(\a_0)[]$. The statement now follows from the commutativity of (\ref{homologydia}). 
\end{proof}

\begin{rem}
\label{thetaremark}
If $Z$ and $W$ are closed subsets of $\Sing(Q/f)$ that satisfy $Z \cap W = \{\fm\}$, then, 
from \eqref{E1116}, we obtain the pairing
$$
HH_*(mf^Z(Q,f)) \times HH_*(mf^W(Q,-f)) \xra{\star} HH_*(mf^{\fm}(Q,0)).
$$
By Proposition \ref{(1)}, given $X \in mf^Z(Q,f)$ and $Y \in mf^W(Q,-f)$, the composition
\begin{align*}
H_*(\End(X)) \times H_*(\End(Y)) & \to HH_*(mf^Z(Q,f)) \times HH_*(mf^W(Q,-f)) \\
&  \xra{\star} HH_*(mf^{\fm}(Q,0)) \\
& \xra{\trace} k
\end{align*}
  sends a pair of endomorphisms $(\a, \b)$ to $\tr(H_0(\a \otimes \b)) - \tr(H_1(\a \otimes \b))$. In particular, it sends $(\id_X, \id_Y)$ to 
$$
\theta(X,  Y) := \dm_k H_0(X \otimes Y) - \dm_k H_1(X \otimes Y).
$$
\end{rem}

Recall from Subsection \ref{calculation}
the folded Koszul complex $K$ and the exterior algebra $\Lambda \subseteq \End_{mf^\fm(Q_\fm, 0)}(K)$.
Denote by $\eta: \Lambda \to k$ the augmentation map that sends $e_i^*$ to $0$.

\begin{prop} \label{prop1116}
The composition
\begin{equation}
\label{lamtrace}
HH_*(\Lambda)  \xra{(\ref{lambdaqi})} HH_*(mf^\fm(Q_\fm,0)) \xra{\trace} k
\end{equation}
coincides with
\begin{equation}
\label{augmentation}
HH_*(\Lambda) \xra{HH_*(\eta)} HH_*(k) \xra{\cong} k,
\end{equation}
where the second map in (\ref{augmentation}) is the canonical isomorphism.
In particular, if $\a_0[\a_1 | \dots | \a_n]$ is a cycle in $HH(\Lambda)$, where $n > 0$, the map (\ref{lamtrace}) sends $\a_0[\a_1 | \dots | \a_n]$ to 0.
\end{prop}

\begin{proof} 
If $C$ is a $\Z$-graded complex, denote its $\Z/2$-folding by $\Fold(C)$. Similarly, given a differential $\Z$-graded category $\cC$, define a differential $\Z/2$-graded category $\on{Fold}(\cC)$ with the same objects as $\cC$ and morphism complexes given by taking the $\Z/2$-foldings of the morphism complexes of $\cC$. In this proof, we use the notation $HH^\Z( - )$ (resp. $HH^{\Z/2}( - )$) to denote the Hochschild complex of a differential $\Z$-graded (resp. $\Z/2$-graded) category. We observe that, if $\cC$ is a differential $\Z$-graded category,
\begin{equation}
\label{foldhom}
\Fold(HH_*^\Z(\cC)) = HH_*^{\Z/2}(\Fold(\cC)).
\end{equation}

Let $\Perf^{\fm}(Q)$ denote the dg-category of perfect complexes of $Q$-modules with support in $\{\fm\}$, and let $\Perf_\Z(k)$ denote the differential $\Z$-graded category of complexes of (not necessarily finite dimensional) $k$-vector spaces with finite dimensional total homology. As in the $\Z/2$-graded case, there is an isomorphism
$$
\widetilde{v} : k \xra{\cong} HH_*^\Z(\Perf_\Z(k)),
$$
where $k$ is concentrated in degree 0, given by $a \mapsto a[]$.

Let $\widetilde{K}$ denote the $\Z$-graded Koszul complex on the regular system of parameters $x_1, \dots, x_n$ for $Q_\fm$ chosen in Subsection \ref{calculation}, so that the $\Z/2$-folding of $\widetilde{K}$ is $K$. Similarly, denote by $\widetilde{\Lambda}$ the subalgebra (with trivial differential) of $\End(\widetilde{K})$, defined in the same way as $\Lambda$, so that the $\Z/2$-folding of $\widetilde{\Lambda}$ is $\Lambda$. Notice that every $\a_i$ appearing in our cycle $\a_0[\a_1 | \dots | \a_n]$ can be considered as an element of $\widetilde{\Lambda}$.

We consider the composition
\begin{equation}
\label{Zgraded}
HH_*^\Z(\widetilde{\Lambda}) \to HH_*^\Z(\End(\widetilde{K})) \to HH_*^\Z(\Perf^\fm(Q)) \to HH_*^\Z(\Perf_\Z(k)) \xra{(\widetilde{v})^{-1}} k
\end{equation}
of maps of $\Z$-graded $k$-vector spaces. We claim (\ref{Zgraded}) coincides with the composition
\begin{equation}
\label{Zgradedaug}
HH^\Z_*(\widetilde{\Lambda}) \to HH^\Z_*(k) \xra{\cong} k,
\end{equation}
where the first map is induced by the augmentation map $\widetilde{\Lambda} \to k$. We need only check this in degree 0. $HH^\Z_0(\widetilde{\Lambda})$ is a 1-dimensional $k$-vector space generated by $\id_K[]$. The map (\ref{Zgradedaug}) sends $\id_K[]$ to 1, and, by (the $\Z$-graded version of) Lemma \ref{(1)}, the map (\ref{Zgraded}) does as well. 

Applying $\Fold( - )$ to (\ref{Zgraded}), and using (\ref{foldhom}),
we arrive at a composition
$$
HH_*^{\Z/2}(\Lambda) \to HH_*^{\Z/2}(\Fold(\Perf^\fm(Q))) \to k
$$
of maps of $\Z/2$-graded complexes of $k$-vector spaces, which may be augmented to a commutative diagram
\begin{equation}
\label{splittri}
\xymatrix{
HH_*^{\Z/2}(\Lambda)  \ar[r] \ar[dr]_-{(\ref{lambdaqi})} & HH_*^{\Z/2}(\Fold(\Perf^\fm(Q))) \ar[d] \ar[r]  & k  \\
& HH_*^{\Z/2}(mf^\fm(Q, 0)). \ar[ru]_-{\trace} &
}
\end{equation}
On the other hand, applying $\Fold( - )$ to (\ref{Zgradedaug}), and once again applying (\ref{foldhom}), we get the map (\ref{augmentation}). 
\end{proof}

\begin{lem} \label{lem526a}
  Suppose $Q$ and $Q'$ are regular $k$-algebras, and $\fm \subseteq Q$, $\fm' \subseteq Q'$ are $k$-rational maximal ideals. Let 
 $g: Q \to Q'$ be a $k$-algebra map such that $g^{-1}(\fm') = \fm$, the induced map $Q_{\fm} \to Q'_{\fm'}$ is flat, and $g(\fm) Q'_{\fm'} = \fm' Q'_{\fm'}$.
  Then  $g$ induces a quasi-isomorphism
  $$
  g_*: HH(mf^{\fm}(Q_\fm,0)) \xra{\simeq} HH(mf^{\fm'}(Q'_{\fm'},0)),
  $$
  and
  $$
  \trace_{Q'_{\fm'}} \circ g_* = \trace_{Q_\fm}.
  $$
  \end{lem}
  
  \begin{proof} Let $\widehat{Q}$ (resp. $\widehat{Q'}$) denote the $\fm$-adic (resp. $\fm'$-adic) completion of $Q$ (resp. $Q'$). The assumptions on $g$ imply it induces an isomorphism $\widehat{Q} \xra{\cong} \widehat{Q'}$. The
    first assertion follows since the canonical maps
    $$
    HH_*(mf^{\fm}(Q_\fm,0)) \to HH_*(mf^{\fm}(\widehat{Q},0))
    $$
    and 
    $$
    HH_*(mf^{\fm'}(Q'_{\fm'},0)) \to HH_*(mf^{\fm}(\widehat{Q'},0))
    $$ 
    are isomorphisms by Remark \ref{moritaremark}.

As for the second assertion, let $n = \dim(Q_\fm)$, choose a regular system of parameters $x_1, \dots, x_n$ of $Q_\fm$, and construct the exterior algebra
    $\Lambda$ using this system of parameters, as in Subsection \ref{calculation}. The hypotheses ensure that 
$g(x_1), \dots, g(x_n)$ form a regular system of parameters for $Q'_{\fm'}$, and we 
let $\Lambda'$ be the associated exterior algebra.
  We have a commutative diagram 
    $$
    \xymatrix{
    HH_*(\Lambda) \ar[d]^-{\cong} \ar[r]^-{\cong} & HH_*(\Lambda') \ar[d]^-{\cong} \\
    HH_*(mf^{\fm}(Q_\fm,0)) \ar[r] & HH_*(mf^{\fm'}(Q'_{\fm'},0)), \\
    }
    $$
    where the vertical isomorphisms are as in Lemma \ref{lem529}. By Proposition \ref{prop1116}, it now suffices to observe that the composition
    $$
    HH_*(\Lambda) \xra{\cong} HH_*(\Lambda') \to k,
    $$
    where the second map is induced by the augmentation $\Lambda' \to k$, coincides with the map induced by the augmentation $\Lambda \to k$. 
      \end{proof}

     \begin{lem} \label{lem528c}
       Suppose $Q$, $Q'$ are essentially smooth $k$-algebras and $\fm' \subseteq Q'$, $\fm'' \subseteq Q''$ are $k$-rational maximal ideals.
       Set $Q = Q' \otimes_k Q''$ and $\fm = \fm' \otimes_k Q'' + Q' \otimes_k \fm''$. Then $Q$ is an essentially smooth $k$-algebra, $\fm$ is
a $k$-rational maximal ideal of $Q$,  and the diagram
     $$
     \xymatrix{
       HH_*(mf^{\fm'}(Q'_{\fm'},0)) \otimes_k  HH_*(mf^{\fm''}(Q''_{\fm''},0))  \ar[r]^-{\tstar} \ar[d]_-{\trace \otimes \trace} 
   &         HH_*(mf^{\fm}(Q_\fm,0))  \ar[d]^-{\trace} \\
 k \otimes_k k \ar[r]^-\cong & k
}
$$
commutes.
\end{lem}

\begin{proof}The first two assertions are standard facts. As for the final one, let $n'$ and $n''$ denote the dimensions of $Q'_{\fm'}$ and $Q''_{\fm''}$, resp. Choose regular systems of
parameters $x_1, \dots, x_{n'}$ and $y_1, \dots, y_{n''}$ of $Q'_{\fm'}$ and $Q''_{\fm''}$, resp., so that $x_1, \dots, x_{n'}, y_1, \dots, y_{n''}$ form a regular system of
parameters of $Q_\fm$.
As in the proof of Lemma \ref{lem526a}, let $\Lambda, \Lambda'$, and $\Lambda''$ be exterior algebras associated to these systems of parameters, as constructed in Subsection \ref{calculation}.
By Lemma \ref{lem112}, we have a commutative square
$$
\xymatrix{
HH_*(\Lambda') \otimes_k HH_*(\Lambda'') \ar[r]^-{\cong} \ar[d]^-{\tstar} &  HH_*(mf^{\fm'}(Q'_{\fm'},0)) \otimes_k  HH_*(mf^{\fm''}(Q''_{\fm''},0)) \ar[d]^-{\tstar}\\
HH_*(\Lambda) \ar[r]^-{\cong} 
  &         HH_*(mf^{\fm}(Q_\fm,0)),
   } 
$$
where the horizontal isomorphisms are as in Lemma \ref{lem529}. By Proposition \ref{prop1116}, it now suffices to observe that the composition
$$
\Lambda \xra{\cong} \Lambda' \otimes_k \Lambda'' \to k,
$$
where the second map is the tensor product of the augmentations, coincides with the augmentation $\Lambda \to k$.
 \end{proof}

\subsection{The canonical pairing on Hochschild homology}
\label{canonical}
A $k$-linear differential $\Z/2$-graded category $\cC$ is called {\em proper} if,
for all pairs of objects $(X,Y)$, $\dim_k H_i \Hom_{\cC}(X,Y) < \infty$ for $i = 0, 1$.

\begin{defn} For a proper differential $\Z/2$-graded category $\cC$, the {\em canonical pairing for Hochschild homology} is the map
$$
\eta_\cC(-,-): HH_*(\cC) \otimes_k HH_*(\cC) \to k
$$
given by the composition
$$
\begin{aligned}
HH_*(\cC) \otimes_k HH_*(\cC) 
& \xra{\id \otimes \Phi} HH_*(\cC) \otimes_k HH_*(\cC^\op) \\
& \xra{\tstar} HH_*(\cC \otimes_k \cC^\op) \\
& \xra{HH((X,Y) \mapsto \Hom_\cC(Y,X))} HH_*(\Perf_{\Z/2}(k)) \\
& \xla{\cong} k  \\
\end{aligned}
$$
where $\Phi$ is the map defined in \eqref{E1215c}.
\end{defn}

When $\Sing(Q/f) = \{\fm\}$, $mf(Q, f)$ is proper, so we have the canonical pairing
$$
\eta_{mf}: HH_*(mf(Q,f)) \otimes_k HH_*(mf(Q,f)) \to k.
$$

\begin{lem} \label{lem1119t}
When $\Sing(Q/f) = \{\fm\}$, $\eta_{mf}$ coincides with the pairing given by the composition
$$
\begin{aligned}
HH_*(mf(Q,f)) \otimes_k HH_*(mf(Q,f)) 
& \xra{\id \otimes \Psi } HH_*(mf(Q,f)) \otimes_k HH_*(mf(Q,-f))  \\
& \xra{\star} HH_*(mf^{\fm}(Q,0))  \\
& \xra{\trace} k,  \\
\end{aligned}
$$
where $\Psi$ is defined in \eqref{E1215d}. 
\end{lem}

\begin{proof}

By Lemma \ref{lem112}, there is a commutative square
$$
\xymatrix{
HH_*(mf(Q, f)) \otimes_k HH_*(mf(Q, f)^{\op}) \ar[rr]^-{HH(\id) \otimes HH(D)} \ar[d]^-{- \tstar - } && HH_*(mf(Q, f)) \otimes_k HH_*(mf(Q, -f)) \ar[d]^-{ - \tstar - } \\
HH_*(mf(Q, f) \otimes_k mf(Q, f)^{\op}) \ar[rr]^-{HH(\id \otimes D)} && HH_*(mf(Q, f) \otimes_k mf(Q, -f)),
}
$$
where $D$ is the dg-functor defined in Subsection \ref{duals}. Therefore, it suffices to show the composition
\begin{align*}
HH_*(mf(Q, f) \otimes_k mf(Q, f)^{\op}) & \xra{ 1 \otimes D} HH_*(mf(Q, f) \otimes_k mf(Q, -f)) \\ 
& \xra{\on{can}} HH_*(mf^\fm(Q, 0)) \\
& \xra{\on{Forget}} HH_*(\Perf_{\Z/2}(k))
\end{align*}
coincides with the map induced by the dg-functor
$$
mf(Q, f) \otimes_k mf(Q, f)^{\op} \to \Perf_{\Z/2}(k)
$$
given by $(X, Y) \mapsto \Hom_{mf}(Y, X)$, and this is clear. 
\end{proof}

\subsection{The residue map}
\label{residuemap}

Assume that
$Q$ is an essentially smooth $k$-algebra and $\fm$ is a $k$-rational maximal ideal of $Q$. Let $n$ be the Krull
dimension of $Q_\fm$.
In this subsection, we recall the definition of Grothendieck's residue map
$$
\Gres: H^n_\fm(\Omega^n_{Q_\fm/k}) \to k
$$
and some of its properties. 
Recall from Subsection \ref{sec:HKR} that for any system of parameters $x_1, \dots, x_n$ of $Q_\fm$,
we have a canonical isomorphism
\begin{equation} \label{E93}
H^n_\fm(\Omega^n_{Q_\fm/k}) \cong H^n(\cC(x_1, \dots, x_n) \otimes_{Q_\fm} \Omega^n_{Q_\fm/k} ).
\end{equation}
We will temporarily use $\Z$-gradings and index things cohomologically, using superscripts. In particular, $\Omega^\bu_{Q_\fm/k}$ is a graded $Q_\fm$-module with
$\Omega^j_{Q_\fm/k}$ declared to have cohomological degree $j$.

We introduce some 
notation that will be convenient when computing with the augmented \v{C}ech complex. First form 
the exterior algebra over $Q_\fm[1/x_1, \dots, 1/x_n]$ on (cohomological) degree $1$ generators $\a_1, \dots \a_n$, and make it a complex with differential given as 
left multiplication by the degree $1$ element $\sum_i \a_i$. We identify  $\cC(x_1, \dots, x_n)$ as the subcomplex whose degree $j$ component is
$$
\bigoplus_{i_1 < \cdots < i_j} Q_\fm\left[\frac{1}{x_{i_1} \cdots x_{i_j}}\right] \a_{i_1} \cdots \a_{i_j}.
$$

Define 
$$
E(x_1, \dots, x_n): = 
\frac{Q_\fm[1/x_1, \dots, 1/x_n]}{\sum_j Q_\fm[1/x_1, \dots, \widehat{1/x_j}, \dots, 1/x_n]}.
$$
Since $x_1, \dots, x_n$ is a regular sequence, there is an isomorphism
$$
E(x_1, \dots, x_n) \xra{\cong} H^n(\cC(x_1, \dots, x_n) )
$$
sending $\overline{g}$ to $\overline{g \a_1 \cdots \a_n}$ for $g \in Q_\fm[1/x_1, \dots, 1/x_n]$.
Using that $\Omega^n_{Q_\fm/k}$ is a flat $Q_\fm$-module, we obtain the isomorphism
\begin{equation} \label{E528}
   H^{n}(\cC(x_1, \dots, x_n) \otimes_{Q_\fm} \Omega^n_{Q_\fm/k})
\cong E(x_1, \dots, x_n)  \otimes_{Q_\fm}  \Omega^n_{Q_\fm/k}.
\end{equation}   
Every element of $E(x_1, \dots, x_n)  \otimes_{Q_\fm}  \Omega^n_{Q_\fm/k}$ is a sum of terms of the form
$$
\frac{1}{x_1^{a_1} \cdots x_n^{a_n}} \otimes \omega
$$
with $a_i \geq 1$ and $\omega \in \Omega^n_{Q_\fm/k}$, and this element corresponds to
\begin{equation} \label{E824}
\frac{\a_1 \cdots \a_n}{x_1^{a_1} \cdots x_n^{a_n}} \otimes \omega \in H^{n}(\cC(x_1, \dots, x_n) \otimes_{Q_\fm} \Omega^n_{Q_\fm/k})
\end{equation}
under the isomorphism  \eqref{E528}.

\begin{defn}  \label{def825b}
Given a system of parameters $x_1, \dots, x_n$ for $Q_\fm$, integers $a_i \geq 1$ for each $1 \leq i \leq n$, and an $n$-form  
$\omega \in \Omega^n_{Q_\fm/k}$,  the {\em generalized fraction}
$$
\left[\frac{\omega}{x_1^{a_1}, \dots, x_n^{a_n}}\right] \in H^n_\fm(\Omega^n_{Q_\fm/k})
$$
is the class corresponding to the element in \eqref{E824} under the canonical isomorphism \eqref{E93}. 
\end{defn}




% ...

To define Grothendieck's residue map, we now assume $x_1, \dots, x_n$ is a \emph{regular} system of parameters. Since $\fm$ is $k$-rational, the $\fm$-adic completion $\widehat{Q}$ of $Q$ is isomorphic to the ring of formal power series
$k[[x_1, \dots, x_n]]$, and  a basis for $E(x_1, \dots, x_n)$ as a $k$-vector space is given by the set 
$\{\frac{1}{x_1^{a_1} \cdots x_n^{a_n}} \mid a_i \geq 1\}$. We also have that $\Omega^n_{Q_\fm/k}$ is a free $Q_\fm$-module of rank one spanned by $dx_1 \cdots dx_n$. 
It follows that the set 
$$
\left\{  \left[\frac{dx_1 \cdots dx_n}{x_1^{a_1}, \cdots, x_n^{a_n}}\right] \mid a_i \geq 1\right\}
$$
is a $k$-basis of
$H^n_\fm(\Omega^n_{Q_\fm/k})$. 

\begin{defn} \label{def825}
Grothendieck's residue map $\Gres: H^n_{\fm}(\Omega^n_{Q/k}) \to k$ is the unique  $k$-linear map 
such that,
if $x_1, \dots, x_n$ is a regular system of parameters of $Q_\fm$, then 
\begin{equation} \label{E527}
\Gres \left[  \frac{dx_1 \cdots dx_n} {x_1^{a_1}, \cdots, x_n^{a_n}} \right] =
\begin{cases} 
1 & \text{if $a_i = 1$ for all $i$, and } \\
0 & \text{otherwise.}
\end{cases}  
\end{equation}
\end{defn}
See \cite[Theorem 5.2]{kunz} for a proof that this definition is independent of the choice of $x_1, \dots, x_n$.

We now revert to the $\Z/2$-grading used throughout most of this paper. In particular, we regard $\Omega^\bu_{Q_\fm/k}$ 
as a $\Z/2$-graded $Q_\fm$-module with $\Omega^j_{Q_\fm/k}$ located in degree $j \pmod 2$, and we use subscripts to indicate degrees. 

\begin{defn}
The {\em residue map} for the $\Z/2$-graded $Q_\fm$-module $\Omega^\bu_{Q_\fm/k}$ is the map
$$
\on{res}= \on{res}_{Q,\fm}: H_{2n} \R\Gamma_\fm(\Omega^\bu_{Q_\fm/k})  \to k,
$$
defined as the composition
$$
H_{2n} \R\Gamma_\fm(\Omega^\bu_{Q_\fm/k})  
\onto H_{2n} \R\Gamma_\fm(\Sigma^{-n} \Omega^n_{Q_\fm/k}) \cong  H_\fm^n(\Omega^n_{Q_\fm/k}) \xra{\Gres} k,
$$
where the first map is induced by the canonical projection $\Omega^\bu_{Q_\fm/k} \onto \Sigma^{-n} \Omega^n_{Q_\fm/k}$.
\end{defn}

We will need the following two properties of the residue map:

\begin{lem}  \label{lem531} 
  Suppose $Q$ and $Q'$ are essentially smooth $k$-algebras and $\fm \subseteq Q$, $\fm' \subseteq Q'$ are $k$-rational maximal ideals. Let 
 $g: Q \to Q'$ be a $k$-algebra map such that $g^{-1}(\fm') = \fm$, the induced map $Q_{\fm} \to Q'_{\fm'}$ is flat, and $g(\fm) Q'_{\fm'} = \fm' Q'_{\fm'}$. Then  $Q_\fm$ and $Q'_{\fm'}$ have the same Krull dimension, say $n$; $g$ induces an isomorphism
  $$
  g_*: H_{2n} \R\Gamma_\fm (\Omega^\bu_{Q_\fm/k}) \xra{\cong}  H_{2n} \R\Gamma_{\fm'} (\Omega^\bu_{Q'_{\fm'}/k})
    $$
    of $k$-vector spaces; and we have
    $$
    \res_{Q',\fm'} \circ g_* = \res_{Q,\fm}.
    $$
  \end{lem}

  \begin{proof} 
Let $x_1, \dots, x_n$ be a regular
system of parameters for $Q_\fm$, and set $x_i' = g(x_i)$. 
    The assumptions on $g$ give that $x'_1, \dots, x'_n$ form
    a regular system of parameters for $Q'_{\fm'}$, and hence the induced map on completions is an isomorphism. The
first two assertions follow.

The map $E(x_1, \dots, x_n)  \otimes_{Q_\fm} \Omega^n_{Q_\fm/k} 
\to E(x_1', \dots, x_n') \otimes_{Q'_{\fm'}}
    \Omega^n_{Q'_{\fm'}/k}$ induced by $g$ sends
    $\frac{\a_1 \cdots \a_n}{x_1^{a_1} \cdots x_n^{a_n}}  \otimes dx_1 \cdots dx_n$
    to the expression obtained by substituting $x_i'$ for $x_i$, and thus
$$
g_*\left[\frac{dx_1 \cdots dx_n}{x_1^{a_1}, \dots, x_n^{a_n}}\right]
=\left[\frac{dx'_1 \cdots dx'_n}{(x'_1)^{a_1}, \dots, (x'_n)^{a_n}}\right].
$$
The equation $\res_{Q',\fm'} \circ g_* = \res_{Q,\fm}$ follows from \eqref{E527}.
  \end{proof}

  \begin{lem} \label{lem531b} Let $(Q', \fm')$, $(Q'', \fm'')$, and $(Q, \fm) = (Q' \otimes_k Q'', \fm' \otimes_k Q'' + Q' \otimes_k \fm'')$ be as in Lemma \ref{lem528c}. Set $m = \dm(Q')$ and $n = \dm(Q'')$. 
 The diagram
     $$
     \xymatrix{
       H_{2m} \R \G_\fm(\Omega^\bu_{Q'_{\fm'}/k}) \otimes_k       H_{2n} \R \G_{\fm''}(\Omega^\bu_{Q''_{\fm''}/k}) \ar[r]^-{ \smsh} \ar[d]_{\res_{Q',\fm'} \otimes \res_{Q'',\fm''}} 
   &  H_{2m+2n} \R \G_{\fm}(\Omega^{\bu}_{Q_\fm/k})  \ar[d]^{\res_{Q,\fm}} \\
 k \otimes_k k \ar[r]^\cong & k
}
$$
commutes up to the sign $(-1)^{mn}$.
\end{lem}

\begin{proof} It suffices to prove the analogous diagram given by replacing 
$\Omega^\bu_{Q'_{\fm'}/k}$ and
$\Omega^\bu_{Q''_{\fm''}/k}$ with
$\Omega^m_{Q'_{\fm'}/k}$ and
$\Omega^n_{Q''_{\fm''}/k}$ commutes. 
Let $x'_1, \dots, x'_m$ and $x''_1, \dots, x''_n$ be
regular systems of parameters for $Q'_{\fm'}$ and $Q''_{\fm''}$.
  Then, upon identifying $x_i'$ and $x_j''$ with the elements $x_i' \otimes 1$ and $1 \otimes x_i''$ 
  of $Q_\fm$, the sequence $x_1', \dots, x_m', x_1'', \dots, x_n''$ forms a regular
system of parameters for $Q_\fm$. We use these three regular systems of a parameters
to identify
  $H_{2m} \R\G_{\fm'}( \Omega^m_{Q'_{\fm'/k}})$ with $H_{2m} (\cC(x'_1, \dots, x'_m) \otimes_{Q'_{\fm'}}\Omega^m_{Q'_{\fm'}/k}) $ and similarly for $Q''$ and $Q$. 
Under these identifications, the map labelled $\smsh$ in the diagram sends
  $$
  \frac{\a'_1 \cdots \a'_m}{x'_1 \cdots x'_m}  \otimes dx'_1 \cdots dx'_m  \otimes
 \frac{\a''_1 \cdots \a''_n}{x''_1 \cdots x''_n}  \otimes  dx''_1 \cdots dx''_n 
  $$
  to
  $$
(-1)^{mn}   \frac{\a'_1 \cdots \a'_m\a''_1 \cdots \a''_n}{x'_1 \cdots x'_mx''_1 \cdots x''_m} \otimes dx'_1 \cdots dx'_m dx''_1 \cdots dx''_n   ,
  $$
  with the sign arising since the $dx'_i$'s and $\a''_j$'s have odd degree. The result now follows from Definition \ref{def825b} and \eqref{E527}.
  \end{proof}

\subsection{The residue pairing}

We assume $Q$, $k$  and $\fm$ are as in Subsection \ref{residuemap}. All gradings in this section are $\Z/2$-gradings. Fix $f \in Q$, and assume  $\Sing(f: \Spec(Q) \to \A^1_k) = \{\fm\}$. Then the canonical map
$$
(\Omega_{Q/k}^\bu, -df) \to (\Omega^\bu_{Q_\fm/k}, -df)
$$
is a quasi-isomorphism, and the only non-zero homology module is
$$
\frac{\Omega^n_{Q/k}}{df \smsh \Omega^{n-1}_{Q/k}}
\cong
\frac{\Omega^n_{Q_\fm/k}}{df \smsh \Omega^{n-1}_{Q_\fm/k}},
$$ 
located in degree $n := \dm(Q_\fm)$. Choose a regular system of parameters 
$$
x_1, \dots, x_n \in \fm Q_\fm.
$$ 
Then $dx_1, \dots, dx_n$ forms a $Q_\fm$-basis for $\Omega^1_{Q_\fm/k}$, and we write 
$$
\del_1, \dots, \del_n \in \Der_k(Q_\fm) = \Hom_{Q_\fm}(\Omega^1_{Q_\fm/k}, Q_\fm)
$$ 
for the associated dual basis. 
Set $f_i = \del_i(f)$. The sequence $f_1, \dots, f_n$ forms a system of parameters for $Q_\fm$. 
For example, when $Q_\fm = k[x_1, \dots, x_n]_{(x_1, \dots,   x_n)}$, we have $\del_i = \del/\del x_i$, so that $f_i = \del f/\del {x_i}$.

\begin{defn} With the notation of the previous paragraph,  the {\em residue pairing} is the map
$$
\langle -,- \rangle_{\on{res}} : \frac{\Omega^n_{Q/k}}{df \smsh \Omega^{n-1}_{Q/k}} \times \frac{\Omega^n_{Q/k}}{df \smsh \Omega^{n-1}_{Q/k}} \to k
$$
that sends a pair $(gdx_1 \cdots dx_n , h dx_1 \cdots dx_n)$ to  $\Gres \left[\frac{gh dx_1 \cdots dx_n}{f_1, \dots, f_n}\right]$. 
\end{defn}

\begin{prop} \label{prop1129}
The residue pairing coincides with the composition
$$
\begin{aligned}
 \frac{\Omega^n_{Q/k}}{df \smsh \Omega^{n-1}_{Q/k}} \times \frac{\Omega^n_{Q/k}}{df \smsh \Omega^{n-1}_{Q/k}} 
 & = H_n(\Omega^\bu_{Q/k}, -df) \times H_n(\Omega^\bu_{Q/k}, -df) \\
 & \xra{\cong} H_n(\Omega^\bu_{Q_\fm/k}, -df) \times H_n(\Omega^\bu_{Q_\fm/k}, -df) \\
 & \xra{\id \times (-1)^n} H^n(\Omega^\bu_{Q_\fm/k}, -df) \times H^n(\Omega^\bu_{Q_\fm/k}, df) \\ 
 & \xla{\cong} H_n \R\Gamma_{\fm} (\Omega^\bu_{Q_\fm/k}, -df) \times H_n(\Omega^\bu_{Q_\fm/k}, df) \\
 & \xra{\text{K\"unneth}} H_{2n}(\R\Gamma_{\fm}(\Omega^\bu_{Q_\fm/k}, -df) \otimes_{Q_\fm} (\Omega^\bu_{Q_\fm/k}, df)) \\ 
& \xra{\smsh} H_{2n} \R\Gamma_{\fm} (\Omega^\bu_{Q_\fm/k}, 0) \\
& \xra{\on{res}} k. \\
\end{aligned}
$$
In particular, it is well-defined and independent of the choice of regular system of parameters. 
\end{prop}

\begin{proof}
We need a formula for the inverse of the canonical isomorphism
\begin{equation}
\label{supportiso}
H_n \R \G_\fm(\Omega^\bu_{Q_\fm/k}, -df) \xra{\cong} H_n(\Omega^\bu_{Q_\fm/k}, -df).
\end{equation}
Since the isomorphism is $Q_\fm$-linear, we just need to know where the inverse sends $dx_1 \wedge \cdots \wedge dx_n$. Note that $\cC(x_1, \dots, x_n) \otimes_{Q_\fm}\Omega^\bu_{Q_\fm/k}$ is a graded-commutative $Q_\fm$-algebra (but not a dga), 
and the differential is left multiplication by $\sum_i \a_i -  f_i dx_i$. Observe that the element
\begin{align*}
\omega &:= (-\frac{1}{f_1} \a_1 + dx_1) (-\frac{1}{f_2} \a_2 + dx_2) \cdots (-\frac{1}{f_n} \a_n + dx_n) \\
&= (-1)^n\frac{1}{f_1 \cdots f_n} (\a_1 - f_1 dx_1) (\a_2 - f_2 dx_2) \cdots (\a_n - f_n dx_n) \in   \cC \otimes_{Q_\fm}  (\Omega^\bu_{Q_\fm/k},-df) 
\end{align*}
is a cocycle, and it maps to $dx_1 \wedge \cdots \wedge dx_n \in H_n(\Omega^\bu_{Q_\fm/k}, -df)$ via (\ref{supportiso}). Therefore, the composition
$$
\begin{aligned}
   \frac{\Omega^n_{Q/k}}{df \smsh \Omega^{n-1}_{Q/k}} \times \frac{\Omega^n_{Q/k}}{df \smsh \Omega^{n-1}_{Q/k}} 
  & \xra{\cong} H_n(\Omega_{Q_\fm/k}^\bu, -df) \times H_n(\Omega_{Q_\fm/k}^\bu, -df) \\
  & \xra{\id \times (-1)^n} H_n(\Omega_{Q_\fm/k}^\bu, -df) \times H_n(\Omega_{Q_\fm/k}^\bu, df) \\
& \xla{\cong} H_n(\cC \otimes_{Q_\fm}  (\Omega_{Q_\fm/k}^\bu, -df) )  \times H_n(\Omega_{Q_\fm/k}^\bu, df) \\
& \xra{\text{K\"unneth}}  H_{2n}(\cC \otimes_{Q_\fm}  (\Omega_{Q_\fm/k}^\bu, -df)  \otimes_{Q_\fm} (\Omega_{Q_\fm/k}^\bu, df)) \\
\end{aligned}
$$
sends $(gdx_1 \cdots dx_n , h dx_1 \cdots dx_n)$ to
$$
g \prod_i (-\frac{1}{f_i} \a_i + dx_i) \otimes (-1)^nh dx_1 \wedge \cdots \wedge dx_n.
$$
Under the composition
$$
\begin{aligned}
H_{2n}( \cC \otimes_{Q_\fm} (\Omega_{Q_\fm/k}^\bu, -df)  \otimes_{Q_\fm} (\Omega_{Q_\fm/k}^\bu, df)) 
& \xra{\smsh} H_{2n}(  \cC \otimes_{Q_\fm} (\Omega_{Q_\fm/k}^\bu, 0) )  \\
& \xra{\cong} E \otimes_{Q_\fm} \Omega_{Q_\fm/k}^n   ,
\end{aligned}
$$
this element maps to
$$
\frac{gh}{f_1 \cdots f_n} \otimes dx_1 \wedge \cdots \wedge dx_n,
$$
which is sent to $\Gres \left[\frac{gh dx_1 \cdots dx_n}{f_1, \dots, f_n}\right] \in k$ by the residue map.
\end{proof}

\subsection{Relating the trace and residue maps}

Our goal in this subsection is to prove the following theorem:

\begin{thm} \label{thm112}
Let $k$ be a field of characteristic $0$, $Q$ an essentially smooth
$k$-algebra, and $\fm$ a $k$-rational maximal ideal of $Q$. Then the diagram 
$$
\xymatrix{  
HH_0(mf^{\fm}(Q_\fm,0))  \ar[dr]_{(-1)^{\frac{n(n+1)}{2}} \trace} \ar[rr]^\e && 
H_{2n} \R \Gamma_\fm(\Omega^\bu_{Q_\fm/k})   \ar[dl]^{\on{res}} \\
  & k \\
}
$$
commutes, where $n = \dm(Q_\fm)$.
\end{thm}
Our strategy for proving this theorem is to reduce it to the very special case when $Q = k[x]$ and $\fm = (x)$ and then to prove it in that
case via an explicit calculation.

  \begin{lem} \label{lem524}
    Given a pair $(Q,\fm)$ and $(Q',\fm')$ satisfying the hypotheses of Theorem \ref{thm112},
    suppose there is a $k$-algebra map
    $g: Q \to Q'$  such that $g^{-1}(\fm') = \fm$, the induced map $Q_{\fm} \to Q'_{\fm'}$ is flat, and $\fm Q'_{\fm'} = \fm' Q'_{\fm'}$.
Then    
\begin{enumerate}
\item Theorem \ref{thm112} holds for $(Q, \fm)$ if and only if it holds for $(Q', \fm')$. 
\item Theorem \ref{thm112} holds provided it holds in the special case
    where $Q = k[t_1, \dots, t_n]$ and $\fm = (t_1, \dots, t_n)$. 
    \end{enumerate}
     \end{lem}

     \begin{proof} (1) follows from Lemmas \ref{lem526a} and \ref{lem531} and the naturality of the HKR map $\e$. As for (2), for $(Q, \fm)$ as in Theorem \ref{thm112}, applying (1) to the map $g: Q \to Q_\fm$ allows us to reduce
       to the case when $Q$ is local.
       In this case, let $x_1, \dots, x_n$
       be a regular system of parameters  
for $Q$, define $g: k[t_1, \dots, t_n] \to Q$
       to be the $k$-algebra map sending $t_i$ to $x_i$, and apply (1) to $g$.     \end{proof}

     \begin{lem} \label{lem5242}
       Suppose $Q',Q''$ are essentially smooth $k$-algebras, and $\fm' \subseteq Q',\fm'' \subseteq Q''$ are $k$-rational maximal ideals.
       Let $Q = Q' \otimes_k Q''$ and $\fm = \fm' \otimes_k Q'' + Q' \otimes_k \fm''$.
       If Theorem \ref{thm112} holds for each of $(Q',\fm')$ and $(Q'',\fm'')$, then it also holds for $(Q, \fm$). In particular, the Theorem holds in general if
       it holds for the special case $Q = k[x], \fm = (x)$.
         \end{lem}

         \begin{proof}For brevity, let $HH' = HH_0(mf^{\fm'}(Q'_{\fm'},0))$, $HH'' = HH_0(mf^{\fm''}(Q''_{\fm''},0))$, and $HH_0 = HH_0(mf^{\fm}(Q_\fm,0))$,
and similarly $\R\G' = H_{\dm(Q'_{\fm'})} \R\G_{\fm'}(\Omega^\bu_{Q'_{\fm'}/k})$, etc.
We consider the diagram           
           $$
           \xymatrix{
k \otimes_k k \ar[ddd]^-\cong \ar[rrr]^-= & & & k \otimes_k k\ar[ddd]^-\cong \\
& HH' \otimes_k HH''  \ar[d]^-\tstar  \ar[lu] \ar[r]^-{\e' \otimes \e''} & \R\G' \otimes_k \R\G'' \ar[ru] \ar[d]^-{\smsh}& \\
& HH \ar[r]^-\e \ar[ld]  & \R\G \ar[dr] & \\
k \ar[rrr]^-= &&& k,
}
           $$
           where the diagonal maps are the appropriate trace or residue maps.
           The left and right trapezoids commute by Lemmas \ref{lem528c} and \ref{lem531b},
           the middle square commutes by Proposition \ref{prop112}, the top trapezoid commutes by assumption, and the outer square obviously commutes. It
           follows from (\ref{HHexterior}) and Lemma \ref{lem529} that           $HH' \otimes_k HH'' \xra{\tstar} HH$ is an isomorphism.
A diagram chase now shows that the bottom trapezoid commutes, which gives the first assertion.
The second assertion is an immediate consequence of the first assertion and Lemma \ref{lem524}.
           \end{proof}

\begin{proof}[Proof of Theorem \ref{thm112}] By Lemma \ref{lem5242}, we need only show
 $$
 \res \circ \e = - \trace
 $$ 
 in the case where $Q = k[x]$ and $\fm = (x)$. Let $K$ be the Koszul complex on $x$, considered as a differential $\Z/2$-graded algebra, as in Section \ref{calculation}, and let $\cE = \End_{mf^{(x)}(k[x]_{(x)}, 0)}(K_{(x)})$. Recall from Section \ref{calculation} that $\cE$ is the differential $\Z/2$-graded $Q$-algebra generated by odd degree elements $e, e^*$ satisfying the relations $e^2 = 0 = (e^*)^2$ and $[e, e^*] = 1$, and the differential $d^\cE$ is given by $d^\cE(e) = x$ and $d^\cE(e^*) = 0$. By Lemma \ref{lem529}, we have an isomorphism
$$
k[y] \xra{\cong} HH_0(mf^{(x)}(k[x]_{(x)}, 0)),
$$
where
$$
y \mapsto \id_K [e^*] \in HH(\cE) \subseteq HH(mf^{(x)}(k[x]_{(x)},0)),
$$
and, more generally,
$$
y^j \mapsto j! \id_K[\overbrace{e^*| \cdots | e^*}^j], \text{  for $j \geq 0$.}
$$
As usual, we identify $H_2 \R \Gamma_{(x)}(\Omega^\bu_{k[x]_{(x)}/k})$ with $  \frac{k[x]_{(x)}[x^{-1}]}{k[x]_{(x)}}  \cdot  \a  \otimes_{k[x]_{(x)}} \Omega^1_{k[x]_{(x)}/k}$, where $|\a| = 1$. Theorem \ref{thm112} follows from the calculations
\begin{enumerate}
\item $\res( \frac{\a}{x} \otimes dx ) = 1$,
\item $\res( \frac{\a}{x^i} \otimes dx) = 0$ for all $i > 1$,
\item $\trace(y^0) = 1$, 
\item $\trace(y^j) = 0$ for all $j \geq 1$, and
\item $\e(y^j) = - j! ( \frac{\a}{x^{j+1}} \otimes dx )$ for all $j \geq 0$.
\end{enumerate}
In fact, (1) and (2) follow from the definition of the residue map, and (3) and (4) follow from Propositions \ref{(1)} and \ref{prop1116}, so it remains only to establish (5).

Recall that the map $\e$ is induced by the diagram
\begin{equation} \label{E531b}
  k[y] \xra{\cong}
  HH_0(\cE) \xla{\cong} 
H_2 \R\Gamma_{(x)} HH(\cE) \xra{\R \G_{(x)} \e'} 
H_2 \R\Gamma_{(x)}(\Omega^\bu_{k[x]_{(x)}/k}),
\end{equation}
where $\e'$ denotes the composition
$$
HH(\cE) \xra{(\id, d_K)_*}  HH^{II}(\cE^0) \xra{\e^0} 
\Omega^\bu_{k[x]_{(x)}/k}.
$$
Here, $\cE^0$ is the same as $\cE$, but with trivial differential, $(\id, d_K)$ is a morphism $\cE \to \cE^0$ of curved dga's (with trivial curvature), and $\e^0$ is as defined in \ref{nosupport}.

We need to calculate the inverse of the isomorphism
$H_2 \R\Gamma_{(x)} HH(\cE) \xra{\cong}  HH_0(\cE)$
occuring in \eqref{E531b}.
As usual, we make the identification
$$
\R\Gamma_{(x)} HH(\cE) = HH(\cE) \oplus  HH(\cE)[1/x] \cdot \a.
$$
The differential on the right is $\del := b + \a$, where $\a$ denotes left multiplication by $\a$; note that $\a^2 = 0$. So, for a class $\gamma + \gamma'\a$, we have
$$
\del(\gamma + \a \gamma') = b(\gamma) - b(\gamma')\a + \gamma \a.
$$
With this notation, the quasi-isomorphism $\R\Gamma_{(x)} HH(\cE) \xra{\simeq} HH(\cE)$ is given by setting $\a= 0$. 

For $j \geq 0$, we define 
$$
y^{(j)}  = \frac{1}{j!} y^j = \id_K [\overbrace{e^* | e^* |   \cdots | e^*}^{\text{$j$ terms}}]
$$
and
$$
\omega_j  =  e[\overbrace{e^* | e^* |   \cdots | e^*}^{\text{$j$ terms}}] 
\in  HH(\cE)[1/x].
$$
Then, for $j \geq 0$, we have 
$$
b(\omega_j)  = x y^{(j)} - y^{(j-1)},
$$
where $y^{(-1)} : = 0$, from which we get 
$$
b\left( \frac{1}{x} \omega_j + \frac{1}{x^2} \omega_{j-1} + \cdots + \frac{1}{x^{j+1}} \omega_0\right)  = y^{(j)}.
$$
It follows that, for each $j \geq 0$, the class
$$
y^{(j)} + \alpha \left(\frac{1}{x} \omega_j + \frac{1}{x^2} \omega_{j-1} + \cdots \frac{1}{x^{j+1}} \omega_0\right)
$$
is a cycle in $\R\Gamma_{(x)} HH(\cE)$ that maps to $y^{(j)} \in HH(\cE)$ under the canonical map $\R\Gamma_{(x)}(HH(\cE)) \to HH(\cE)$. We conclude that the inverse of
$$
H_2 \R\Gamma_{(x)}HH(\cE) \xra{\cong} HH_0(\cE) = k[y]
$$
maps $y^j$ to the class of 
$$
\eta_j := y^{j} + j! \alpha \left(\frac{1}{x} \omega_j + \frac{1}{x^2} \omega_{j-1} + \cdots + \frac{1}{x^{j+1}} \omega_0\right)
$$
for each $j  \geq 0$, and hence
$$
\e(y^j) = \R\Gamma_{(x)}\e'(\eta_j).
$$

Recall that $\e'$ sends $\t_0[\t_1| \cdots |\t_n] \in HH(\cE)$ to
$$
\sum (-1)^{j_0+ \cdots +j_n} \frac{1}{(n + J)!} \str(\t_0 (d_K')^{j_0} \t_1' \cdots \t_n' (d_K')^{j_n} ),
$$
where the derivatives are computed relative to any specified flat connection on $K$.
Using the Levi-Civita connection associated to the basis
$\{1, e\}$ of $K$, we get $e' = 0$, $(e^*)'= 0$ and hence $d_K' = - e^* dx$. It follows that
$$
\begin{aligned}
  \e'(\omega_j) & = 0 \text{ for $j \geq 1$,} \\
  \e'(\omega_0) & = \str(e) + \str(e e^* dx), \\
  \e'(y^{(j)}) & = 0 \text{ for $j \geq 1$, and } \\
  \e'(y^{(0)}) & = \str(\id_K) + \str(e^* dx). \\
  \end{aligned}
$$
It is easy to see that $\str(e e^*) = -1$, $\str(e^*) = 0$, $\str(e) = 0$, and $\str(\id_K) = 0$, so that 
$\e'(\omega_0) = - dx$, $\e'(\omega_j) = 0$ for all $j \geq 0$, and $\e'(y^j) = 0$ for all $j$. We obtain
$$
\e(y^j) = \R\Gamma_{(x)}\e' (\eta_j) = - j! (\frac{\a}{x^{j+1}} \otimes dx )
$$
for all $j \geq 0$, as needed.
\end{proof}

\subsection{Proof of the conjecture}


Let $Q = \C[x_1, \dots, x_n]$ and $f \in \fm = (x_1, \dots, x_n) \subseteq Q$, and assume $\fm$ is the only singular point of the morphism $f: \Spec(Q) \to \A^1$. As discussed in the introduction, a result of Shklyarov (\cite[Corollary 2]{Shklyarov:higherresidues}) states that there is a commutative diagram

\begin{equation} \label{E1119d}
\xymatrix{
HH_n(mf(Q,f))^{\times 2} \ar[rr]^-{I_f(0) \times I_f(0)}_-\cong \ar[dr]_-{c_{f} \eta_{mf}} &&  
(\frac{\Omega^n_{Q/k}}{ df \smsh \Omega^{n-1}_{Q/k}})^{\times 2} 
\ar[dl]^-{\langle -,- \rangle_{\on{res}}}\\
& \C, \\
}
\end{equation}
for some constant $c_f$ which possibly depends on $f$. 

\begin{thm} \label{mainthm2}
  Let $k$ be a field of characteristic $0$, $Q$ an essentially smooth $k$-algebra, $\fm$ a $k$-rational maximal ideal, and $f$ an element of
  $\fm$ such that $\fm$ is the only singularity of the morphism $f: \Spec(Q) \to \A^1_k$. Then the  diagram
$$
\xymatrix{
HH_n(mf(Q,f))^{\times 2} \ar[rr]^-{\e \times \e}_-\cong \ar[dr]_-{(-1)^{n(n+1)/2} \eta_{mf}} &&  
(\frac{\Omega^n_{Q/k}}{ df \smsh \Omega^{n-1}_{Q/k}})^{\times 2}
\ar[dl]^-{\langle -,- \rangle_{\on{res}}}\\
& k \\
}
$$
commutes.
\end{thm}

\begin{proof}
  Consider the diagram 
   \begin{equation}
   \label{proofdia}
{
\xymatrix{
HH_n(mf(Q,f)) \times HH_n(mf(Q,f))  \ar[d]^-{\id \times \Psi} \ar[rr]^-{\e \times \e} && H_n(\Omega_Q^\bu, -df) \times H_n(\Omega_Q^\bu, -df)
\ar[d]^-{\id \times (-1)^n}\\
HH_n(mf(Q,f)) \times HH_n(mf(Q,-f)) \ar[rr]^-{\e \times \e} 
\ar[d]^-\star  && H_n(\Omega_Q^\bu, -df) \times H_n(\Omega_Q^\bu, df)
\ar[d]^-{\smsh} \\
HH_{2n}(mf^\fm(Q_\fm,0))  \ar[dr]_-{(-1)^{n(n+1)/2} \trace} \ar[rr]^-\e && H_{2n} \R\Gamma_\fm( \Omega_{Q_\fm}^\bu) \ar[dl]^-{\on{res}} \\
& k. \\
}
}
\end{equation}
The top square commutes by Lemma \ref{lem1129}, the square in the middle  commutes by Corollary \ref{cor1129},
and the triangle at the bottom commutes by Theorem \ref{thm112}.
By Lemma \ref{lem1119t},
the map
$$
HH_n(mf(Q,f)) \times HH_n(mf(Q,f)) \to k
$$
obtained by composing the maps along the left edge of (\ref{proofdia})
is $(-1)^{n(n+1)/2} \eta_{mf}$.
By Proposition \ref{prop1129}, the map
$$
\left(\frac{\Omega^n_{Q/k}}{ df \smsh \Omega^{n-1}_{Q/k}} \right)^{\times 2} = H_n(\Omega_Q^\bu, -df)^{\times 2}  \to k
$$
obtained by composing the maps along the right edge of (\ref{proofdia}) is $\langle -,- \rangle_{\on{res}}$.
\end{proof}

\begin{cor} Conjecture \ref{conj:s} holds. That is, for $f \in \fm = (x_1, \dots, x_n) \subseteq Q = \C[x_1, \dots, x_n]$ such that $\fm$ is the only singularity of the morphism $f: \Spec(Q) \to \A^1_k$,
the unique constant $c_f$ that makes diagram (\ref{E122b})
commute is $(-1)^{n(n+1)/2}$, as predicted by Shklyarov.
\end{cor}

\begin{proof} Under these assumptions, $\e = I_f(0)$ by Lemma \ref{If}.  Theorem \ref{mainthm2} thus implies that 
the value $c_f = (-1)^{n(n+1)/2}$ causes the diagram \eqref{E1119d} to commute. As discussed in the introduction, this uniquely determines the value of $c_f$, and the unique constant $c_f$ which makes diagram \eqref{E1119d} commute is the same as that which makes diagram (\ref{E122b}) commute.
\end{proof}

\section{Recovering Polishchuk-Vaintrob's Hirzebruch-Riemann-Roch formula for matrix factorizations}
\label{PV}
Assume $k$, $Q$, $\fm$, and $f$ are as in the statement of Theorem \ref{mainthm2}. We recall that, given objects $X, Y \in mf(Q, f)$, the \emph{Euler pairing} applied to the pair $(X, Y)$ is given by
$$
\chi(X, Y)  =  \dim_k H_0 \Hom(X, Y) - \dim_k H_1 \Hom(X, Y).
$$
In this final section, we give a new proof of a theorem due to Polishchuk-Vaintrob that relates the Euler pairing to the residue pairing
via the Chern character map. 

The following is an immediate consequence of the commutativity of diagram (\ref{proofdia}) in the proof of Theorem \ref{mainthm2}:

\begin{cor} \label{cor1215}
Let $k$, $Q$, $\fm$, and $f$ be as in the statement of Theorem \ref{mainthm2}, and assume $n = \dm(Q_\fm)$ is even.
Then the triangle
$$
\xymatrix{
  HH_0(mf(Q,f)) \otimes_k  HH_0(mf(Q,-f)) \ar[rr]^-{ \e \otimes_k \e} \ar[dr]
  && \frac{\Omega^n_{Q/k}}{df \smsh \Omega^{n-1}_{Q/k}} \otimes_k
\frac{\Omega^n_{Q/k}}{df \smsh \Omega^{n-1}_{Q/k}}\ar[dl]^-{\langle -,- \rangle_{\on{res}}} \\
& k \\
}
$$
commutes, where the left diagonal map is $(-1)^{n(n+1)/2} \on{trace}\circ (- \star - )$, and $\e$ denotes the composition of the HKR map and the isomorphism
$H_n(\Omega_{Q/k}^\bu, \pm df) \xra{\cong}\frac{\Omega^n_{Q/k}}{df \smsh \Omega^{n-1}_{Q/k}}$.
\end{cor}

Let $X \in mf(Q, f)$. We recall that the \emph{Chern character of $X$}
$$
ch(X) \in HH_0(mf(Q,f)
$$
is the class represented by
$$
\id_X[] \in \End(X) \subseteq HH(mf(Q,f)).
$$
Assume now that $n$ is even. The isomorphism 
$$
\e : HH_0(mf(Q,f)) \xra{\cong} \frac{\Omega^n_{Q/k}}{df \smsh \Omega^{n-1}_{Q/k}}
$$
sends $ch(X)$ to the class
$$
\frac{1}{n!} \str((\delta_X')^n),
$$
where $\delta_X' = [\n, \delta_X]$ for any choice of connection $\n$ on $X$. Abusing notation, we also denote this element of $\frac{\Omega^n_{Q/k}}{df \smsh \Omega^{n-1}_{Q/k}}$ as $ch(X)$. 

For example, if the components of $X$ are free, 
then, upon choosing bases, we may represent $\delta_X$ as a pair of square matrices $(A,B)$ satisfying $AB = f I_r = BA$.  
Using the Levi-Cevita connection associated to this choice of basis, we have
\begin{equation}
\label{freecomp}
ch(X) = \frac{2}{n!} \tr(\overbrace{dA dB \cdots dA dB}^{\text{ $n$ factors}}).
\end{equation}

Recall from Remark \ref{thetaremark} that, for $X \in mf(Q, f)$ and $Y \in mf(Q, -f)$, $\theta(X ,Y)$ is given by 
$$
\dim_k H_0(X \otimes  Y) - \dim_k H_1(X \otimes Y),
$$
and we have
\begin{equation}
\label{thetachern}
\theta(X, Y) = \trace(ch(X) \star ch(Y)).
\end{equation}

\begin{cor} 
\label{PVcor}
Under the assumptions of Corollary \ref{cor1215}, 
\begin{enumerate}
\item If $X \in mf(Q, f)$ and $Y \in mf(Q, -f)$, 
$$
\theta(X, Y) = (-1)^{n \choose 2}\langle ch(X), ch(Y) \rangle_{\on{res}}.
$$
\item If $X, Y \in mf(Q, f)$, 
$$
\chi(X, Y) = (-1)^{n \choose 2}\langle ch(X), ch(Y) \rangle_{\on{res}}.
$$
\end{enumerate}
\end{cor}

\begin{rem}
Corollary \ref{PVcor} (2) is Polishchuk-Vaintrob's Hirzebruch-Riemann-Roch formula for matrix factorizations (\cite[Theorem 4.1.4(i)]{PV}).
\end{rem}

\begin{proof} 
(1) is immediate from Corollary \ref{cor1215} and (\ref{thetachern}). We now prove (2). Without loss of generality, we may assume $Q$ is local, so that the underlying $\Z/2$-graded $Q$-modules of $X$ and $Y$ are free. Given a matrix factorization $(P, \d_P) \in mf(Q ,f)$ written in terms of its $\Z/2$-graded components as 
$$
(\d_1 : P_1 \to P_0, \d_0 : P_0 \to P_1),
$$
we define a matrix factorization $N(P, \d_P) \in mf(Q, -f)$ with components 
$$
(\d_1 : P_1 \to P_0, - \d_0 : P_0 \to P_1).
$$

We have
$$
\langle ch(X), ch(N(Y)) \rangle_{\on{res}} =  (-1)^{{n \choose 2}} \theta(X, N(Y)) =  \chi(X, Y).
$$
The first equality follows from (1), and the second equality follows from \cite[Corollary 8.5]{BW2} and 
 \cite[Proposition 3.18]{BMTW}; note that $(-1)^{{n \choose 2}} = (-1)^{\frac{n}{2}}$, since $n$ is even, and also that the notation $\chi$ in \cite[Proposition 3.18]{BMTW} has a different meaning than it does here. It suffices to show $ch(N(Y)) = (-1)^{\frac{n}{2}}ch(Y)$, and this is clear by (\ref{freecomp}).
\end{proof}


\bibliographystyle{amsalpha}
\bibliography{Bibliography}

\end{document}